\documentclass[10pt,reqno]{amsart}
\usepackage{amssymb,mathrsfs,graphicx}
\usepackage{ifthen}
\usepackage{colortbl}
\definecolor{black}{rgb}{0.0, 0.0, 0.0}
\definecolor{red}{rgb}{1.0, 0.5, 0.5}
\provideboolean{shownotes} 
\setboolean{shownotes}{true} 
\newcommand{\margnote}[1]{
\ifthenelse{\boolean{shownotes}}%
{\marginpar{\raggedright\tiny\texttt{#1}}}%
{}%
}
\newcommand{\hole}[1]{
\ifthenelse{\boolean{shownotes}}%
{\begin{center} \fbox{ \rule {.25cm}{0cm} \rule[-.1cm]{0cm}{.4cm}
\parbox{.85\textwidth}{\begin{center} \texttt{#1}\end{center}} \rule
{.25cm}{0cm}}\end{center}} {} }

\title[Coupled isothermal Euler and isentropic Navier-Stokes equations]{Global classical solutions and large-time behavior of the two-phase fluid model}

\author[Choi]{Young-Pil Choi}
\address[Young-Pil Choi]{\newline Fakult\"at f\"ur Mathematik
    \newline Technische Universit\"at M\"unchen, Boltzmannstr\ss e 3, 85748, Garching bei M\"unchen, Germany}
\email{ychoi@ma.tum.de}

\numberwithin{equation}{section}

\newtheorem{theorem}{Theorem}[section]
\newtheorem{lemma}{Lemma}[section]

\newtheorem{proposition}{Proposition}[section]
\newtheorem{remark}{Remark}[section]

\newcommand{\R}{\mathbb R}
\newcommand{\N}{\mathbb N}
\newcommand{\om}{\Omega}
\newcommand{\ls}{\lesssim}

\newcommand{\T}{\mathbb T}

\newcommand{\mb}{\mathcal B}
\newcommand{\mc}{\mathcal C}
\newcommand{\me}{\mathcal E}
\newcommand{\md}{\mathcal D}
\newcommand{\mh}{\mathcal H}
\newcommand{\ml}{\mathcal L}

\newcommand{\bq}{\begin{equation}}
\newcommand{\eq}{\end{equation}}
\newcommand{\e}{\varepsilon}
\newcommand{\lt}{\left}
\newcommand{\rt}{\right}

\newcommand{\pa}{\partial}
\newcommand{\intt}{\int_{\T^3}}
\newcommand{\wt}{\widetilde}
\newcommand{\sigg}{{\sigma_1,\sigma_2}}
\def\charf {\mbox{{\text 1}\kern-.30em {\text l}}}

\begin{document}
\allowdisplaybreaks

\date{\today}

\subjclass[]{}

\thanks{\textbf{Acknowledgments.} The author was supported by Engineering and Physical Sciences Research Council (EP/K008404/1). This work is supported by the Alexander von Humboldt Foundation through the Humboldt Research Fellowship for Postdoctoral Researchers.
}

\begin{abstract} We study the global existence of a unique strong solution and its large-time behavior of a two-phase fluid system consisting of the compressible isothermal Euler equations coupled with compressible isentropic Navier-Stokes equations through a drag forcing term. The coupled system can be derived as the hydrodynamic limit of the Vlasov-Fokker-Planck/isentropic Navier-Stokes equations with strong local alignment forces. When the initial data is sufficiently small and regular, we establish the unique existence of the global $H^s$-solutions in a perturbation framework. We also provide the large-time behavior of classical solutions showing the alignment between two fluid velocities exponentially fast as time evolves. For this, we construct a Lyapunov function measuring the fluctuations of momentum and mass from its averaged quantities.
\end{abstract}

\maketitle \centerline{\date}


%
%
%
%
\section{Introduction and main results}\label{sec:intro} In this paper, we are concerned with the global existence of a unique classical solutions and the large-time behavior for the compressible isothermal Euler equations coupled with compressible isentropic Navier-Stokes equations. Let $\rho(x,t)$ and $n(x,t)$ be the densities of fluid equations at a domain $(x,t) \in \om \times \R_+$, and let $u(x,t)$ and $v(x,t)$ be the corresponding bulk velocities of $\rho(x,t)$ and $n(x,t)$, respectively. Here we consider two cases that $\om$ is either the periodic domain $\T^3$ or the whole space $\R^3$. Then our coupled hydrodynamic equations read as follows:
\begin{align}\label{f_eq}
\begin{aligned}
&\pa_t \rho + \nabla_x \cdot (\rho u) = 0, \qquad x \in \om, \quad t > 0,\cr
&\pa_t (\rho u) + \nabla_x \cdot (\rho u \otimes u) + \nabla_x \rho = -\rho(u-v),\cr
&\pa_t n + \nabla_x \cdot (n v) = 0,\cr
&\pa_t (nv) + \nabla_x \cdot (nv \otimes v) + \nabla_x p(n) + Lv = \rho(u-v),
\end{aligned}
\end{align}
subject to the initial data
\bq\label{ini_f_eq}
(\rho(x,t),u(x,t),n(x,t),v(x,t))|_{t = 0} = (\rho_0(x),u_0(x),n_0(x),v_0(x)), \quad x \in \om,
\eq
and the boundary conditions
\[
\rho(x,t) \to \rho^\infty \in \R_+, \quad n(x,t) \to n^\infty \in \R_+, \quad u(x,t) \to u^\infty, \quad v(x,t) \to v^\infty, \quad \mbox{as} \quad |x| \to \infty,
\]
if $\om = \R^3$.
Here the pressure $p$ and the Lam\'e operator $L$ are given by
\begin{align}\label{eq_pl}
\begin{aligned}
& p(n) = n^\gamma \quad \mbox{with} \quad \gamma  > 1,\cr
& Lv = -\mu \Delta_x v - (\mu + \lambda)\nabla_x ( \nabla_x \cdot v ) \quad \mbox{with} \quad \mu > 0 \quad \mbox{and} \quad \lambda + 2\mu > 0.
\end{aligned}
\end{align}
Without loss of generality and for the sake of simplicity, throughout the paper, we assume that $\rho^\infty = n^\infty = 1$ and $u^\infty = v^\infty = 0$.
%
%
%
%
\subsection{Formal derivation of the two-phase fluid model \eqref{f_eq}} In this part, we address the formal derivation of the coupled hydrodynamic equations \eqref{f_eq} from kinetic-fluid equations which is a type of Vlasov-Fokker-Planck/compressible Navier-Stokes equations.

More specifically, let $f(x,\xi,t)$ be the distribution function of particles at the position-velocity $(x,\xi)\in \om \times \R^3$ at time $t \in \R_+$, and $n(x,t)$ and $v(x,t)$ be the isentropic compressible fluid density and velocity, respectively. In this situation, we can consider the following kinetic-fluid equations with local alignment and noise forces for the particles to describe the dynamics of particles immersed in the compressible fluid:
\begin{align}\label{eq_kf}
\begin{aligned}
&\pa_t f + \xi \cdot \nabla_x f + \nabla_\xi \cdot ((v - \xi)f) = - \alpha \nabla_\xi \cdot ((u_f - \xi)f) + \sigma \Delta_\xi f, \qquad (x,\xi) \in \om \times \R^3, \quad t > 0,\cr
&\pa_t n + \nabla_x \cdot(nv) = 0,\cr
&\pa_t (nv) + \nabla_x \cdot (nv \otimes v) + \nabla_x p(n) + Lv = \int_{\R^3} (\xi - v)f\,d\xi,
\end{aligned}
\end{align}
where the pressure law $p$ and the operator $L$ are given in \eqref{eq_pl}, and the averaged local velocity $u_f$ is defined by
\[
u_f(x,t) := \frac{\int_{\R^3} \xi f(x,\xi,t) d\xi}{\int_{\R^3} f(x,\xi,t) d\xi}.
\]

Recently, this type of coupled kientic-fluid equations describing the interactions between particles and fluid have received increasing attention due to a number of their applications in the field of, for example, biotechnology, medicine, and in the study of sedimentation phenomena, compressibility of droplets of the spray, and diesel engines, etc \cite{BBKT, Sar, Will}. 

For the system \eqref{eq_kf} without the local alignment forces, i.e., $\alpha = 0$, the global existence of weak solution in a bounded domain with Dirichlet or reflection boundary conditions is obtained in \cite{MV}. In \cite{BD}, the local-in-time existence of classical solutions for the Vlasov/compressible Euler equations is established. For the Vlasov-Fokker-Planck/compressible Euler equations, the global classical solutions are treated in \cite{DL} when the initial data are a small smooth perturbation of constant equilibrium states and in addition the convergence of solutions toward equilibrium is studied. Without the interactions with the fluid, the system \eqref{eq_kf} reduces to the Vlasov-Fokker-Planck equation with the local alignment forces. For this system, global existence of weak solutions is studied in \cite{KMT} and global classical solutions near Maxwellians converging asymptotically to them are constructed in \cite{Choi3}. 

We now take into account a regime where the local alignment and noise forces are strong, i.e., $\alpha = \sigma = \e^{-1}$, and denote the solutions to the system \eqref{eq_kf} with $\alpha = \sigma = \e^{-1}$ by $(f^\e, n^\e, v^\e)$.

Then in the limiting case $\e \to 0$ we can expect the particle distribution function $f^\e(x,\xi,t)$ converges to 
\[
f(x,\xi,t) = \frac{\rho_{f}(x,t)}{(2\pi)^{3/2}}e^{-\frac{|u_{f}(x,t) - \xi|^2}{2}}, 
\]
since the right-hand side of the equation $\eqref{eq_kf}_1$ converges to zero, i.e.,
\[
-\nabla_\xi \cdot ((u_{f^\e} - \xi) f^\e) + \Delta_\xi f^\e \to 0 \quad \mbox{as} \quad \e \to 0,
\]
where 
\[
\rho_{f}(x,t) = \int_{\R^3} f(x,\xi,t)\,d\xi. 
\]
We notice that the continuity equation $\eqref{f_eq}_1$ can be easily obtained by integrating it with respect to $\xi$ together with assuming the convergences $\rho_{f^\e} \to \rho_f$ and $u_{f^\e} \to u_f$ as $\e \to 0$. 

In order to derive the momentum equations $\eqref{f_eq}_2$ formally, we multiply $\eqref{eq_kf}_2$ by $\xi$ and integrating it with respect to $\xi$ to get
$$\begin{aligned}
\frac{d}{dt}\int_{\R^3} \xi f^\e\,d\xi &= \int_{\R^3} \xi \lt(- \nabla_x \cdot(\xi f^\e) - \nabla_\xi \cdot ((v^\e - \xi)f^\e) \rt)\,d\xi- \frac1\e\int_{\R^3} \xi \lt( \nabla_\xi \cdot ((u_{f^\e} - \xi)f^\e) - \Delta_\xi f^\e\rt) \,d\xi\cr
&= - \nabla_x \cdot \int_{\R^3} \xi \otimes \xi \,f^\e\,d\xi + \int_{\R^3} (v^\e - \xi)f^\e\,d\xi\cr
&=:I_1^\e + I_2^\e,
\end{aligned}$$
where we used
\bq\label{eq_zm}
\int_{\R^3} (u_{f^\e} - \xi) f^\e\,d\xi = 0.
\eq
Here $I_i^\e,i=1,2$ are estimated as follows.
$$\begin{aligned}
I_1^\e &= -\nabla_x \cdot \lt( \int_{\R^3} (\xi - u_{f^\e})\otimes (\xi - u_{f^\e})\,f^\e\,d\xi - \int_{\R^3} u_{f^\e}\otimes u_{f^\e}\,f^\e\,d\xi\rt),\cr
I_2^\e &= \rho_{f^\e}(v^\e - u_{f^\e}),
\end{aligned}$$
due to \eqref{eq_zm}. Then, by assuming the appropriate convergences of solutions, we obtain
$$\begin{aligned}
I_1^\e \to I_1 &= -\nabla_x \cdot \lt(\int_{\R^3} (\xi - u_f)\otimes (\xi - u_f)\,\frac{\rho_f}{(2\pi)^{3/2}}e^{-\frac{|u_f - \xi|^2}{2}}\,d\xi \rt)\cr
&\quad  - \nabla_x \cdot(\rho_f u_f \otimes u_f) \int_{\R^3}\frac{1}{(2\pi)^{3/2}}e^{-\frac{|u_f - \xi|^2}{2}}\,d\xi\cr
&= - \nabla_x \rho_f - \nabla_x \cdot(\rho_f u_f \otimes u_f) \cr
I_2^\e \to I_2 &= \rho_f(v-u_f),
\end{aligned}$$
as $\e \to 0$, where we used
\[
\int_{\R^3} (\xi - u_f)\otimes (\xi - u_f)\,\frac{1}{(2\pi)^{3/2}}e^{-\frac{|u_f - \xi|^2}{2}}\,d\xi = \mathbb{I}_{3} \quad\mbox{and} \quad \int_{\R^3}\frac{1}{(2\pi)^{3/2}}e^{-\frac{|u_f - \xi|^2}{2}}\,d\xi = 1.
\]
Here $\mathbb{I}_{3}$ denotes the $3 \times 3$ identity matrix. By combining the above estimates, we derive the desired momentum equations $\eqref{f_eq}_2$. 

\begin{remark}In \cite{CCK}, the kinetic equation $\eqref{eq_kf}_1$ interacting with the viscous incompressible Navier-Stokes equations are considered, and the rigorous derivation of the two-phase fluid model which consists of the isothermal Euler equations and incompressible Navier-Stokes equations is established. This asymptotic analysis is achieved by the relative entropy argument. For the kinetic equation $\eqref{eq_kf}_1$ with a non-local alignement force instead of the drag forcing term, the similar asymptotic analysis is treated in \cite{KMT2}.
\end{remark}


\subsection{Main results}

For the global existence and uniqueness of classical solutions, by setting $h:= \ln \rho$, we reformulate the system \eqref{f_eq} as follows:
\begin{align}\label{main_eq}
\begin{aligned}
&\pa_t h + \nabla h \cdot u + \nabla \cdot u = 0, \quad x \in \om, \quad t >0,\cr
&\pa_t u + u \cdot \nabla u + \nabla h = -(u-v), \cr
&\pa_t n + \nabla_x \cdot ((n+1)v) = 0,\cr
&\pa_t((n+1)v) + \nabla_x \cdot ((n+1)v \otimes v) + \nabla_x p(n+1) + Lv = e^h(u-v)
\end{aligned}
\end{align}
with initial data 
\bq\label{ini_main_eq}
(h(x,t),u(x,t),n(x,t),v(x,t))|_{t=0} =: (h_0(x) = \ln \rho_0(x),u_0(x),n_0(x),v_0(x)), \quad x \in \om.
\eq
Before stating our main result, we define our solution space:
\begin{align*}
\begin{aligned}
\mathcal{I}_s(T):= \Big\{ (h,u,n,v)\, |&\, h \in \mc([0,T];H^s(\om)) \cap\mc^1([0,T];H^{s-1}(\om)) ,
\cr 
&  u\in \mc([0,T];H^{s}(\om)) \cap \mc^1([0,T];H^{s-1}(\om)),  \cr
& n \in \mc([0,T];H^{s}(\om)) \cap\mc^1([0,T];H^{s-1}(\om)),\quad \mbox{and} \cr
& v\in \mc([0,T];H^{s}(\om))\cap L^2(0,T;H^{s+1}(\om)) \Big\}.
\end{aligned}
\end{align*}

\begin{theorem}\label{thm_ext} Let $s > 5/2$. Suppose that the initial data $(h_0,u_0,n_0,v_0)$ satisfy
\begin{align}\label{ass_ini}
\begin{aligned}
&(i)\,\,\,\, 1 + \inf_{x \in \om} n_0(x) > 0,\cr
&(ii)\,\, (h_0,u_0,n_0,v_0) \in H^s(\om) \times H^s(\om) \times H^s(\om) \times H^s(\om).
\end{aligned}
\end{align}
If $\|(h_0,u_0,n_0,v_0)\|_{H^s} \leq \varepsilon_1$ for $\varepsilon_1 >0$ small enough, the Cauchy problem \eqref{main_eq}-\eqref{ini_main_eq} has a unique global classical solution $(h,u,n,v) \in \mathcal{I}_s(\infty)$. 
\end{theorem}

The proof of Theorem \ref{thm_ext} is provided in Section \ref{sec:ap}. The main difficulty is the part of isothermal Euler equations in the system \eqref{main_eq}. Because it is well-known that the compressible Euler equations have the formation of singularities in a finite-time even with smooth initial data. Concerning the issue of development of the singularity, the coupled system \eqref{main_eq} still presents a new challenge to the global existence of classical solutions. Inspired by the author's recent work \cite{Choi1}, we reinterpret the drag forcing term as the relative damping together with using the smoothing effect of viscosity in the compressible Navier-Stokes equations $\eqref{main_eq}_3$-$\eqref{main_eq}_4$ through the drag forcing to prevent the development of the finite-time singularities. Although the drag forcing term does not give the real damping effect, our careful analysis enable us to obtain uniform bounds on the density.

Our second result is about the large-time behavior of classical solutions in the periodic domain, i.e., $\om = \T^3$. For this, we introduce a Lyapunov function measuring the fluctuation of momentum and mass from the corresponding averaged quantities:
\begin{equation}\label{Lyap}
\mathcal{L}(t):= \int_{\T^3} \rho |u-m_c|^2 dx + \int_{\T^3} (\rho - \rho_c)^2\,dx + \int_{\T^3} n| v - j_c|^2 dx + |m_c - j_c|^2 + \int_{\T^3} (n-n_c)^2 dx, 
\end{equation}
where the averaged quantities $m_c, j_c, \rho_c$, and $n_c$ are given by
\begin{equation}\label{def1}
m_c(t) := \frac{\int_{\T^3} \rho u\, dx}{\int_{\T^3} \rho \,dx}, \quad j_c(t) := \frac{\int_{\T^3} n v \,dx}{\intt n\,dx}, \quad \rho_c(t) := \int_{\T^3} \rho \, dx, \quad\mbox{and} \quad n_c(t) := \intt n\,dx.
\end{equation}

\begin{theorem}\label{thm_large} Let $(\rho,u,n,v)$ be any global classical solutions to the system \eqref{f_eq}-\eqref{ini_f_eq}. Suppose that the following conditions hold.
$$\begin{aligned}
&(i) \,\,\,\,\,\, \rho \in [0,\bar \rho], \,\,n \in [0, \bar n], \,\, \rho_c(0), \,\,n_c(0) \in (0,\infty) \quad \mbox{for some} \quad \bar \rho, \bar n >0,\cr
&(ii) \,\,\,\, u,v \in L^\infty(\T^3 \times \R_+),\cr
&(iii)\,\, \mbox{An initial energy } \widetilde E_0 := \intt \rho_0|u_0|^2\,dx + \intt (\rho_0 - \rho_c(0))^2\,dx + \intt n_0|v_0|^2\,dx \cr
&\qquad \qquad \qquad \qquad\qquad \qquad+ \intt (n_0 - n_c(0))^2\,dx \quad \mbox{is sufficiently small.}
\end{aligned}$$
Then we have
\[
\ml(t) \ls \ml_0 e^{-Ct}, \quad t \geq 0,
\]
where $\ml_0 := \ml(0)$ and $C$ is a positive constant independent of $t$. Here $f \ls g$ represents that there exists a positive constant $C>0$ such that $f \leq C g$.
\end{theorem}

A well chosen Lyapunov function $\ml$ in \eqref{Lyap} is very important in the proof of Theorem \ref{thm_large}. Motivated by \cite{Choi1,Choi2}, we first find a temporal energy function $\me$ given in \eqref{tem_e} by taking into account the conservation of total momentum (see Lemma \ref{lem_energy}). On the other hand, we can not get the correct dissipation rate for the convergence from the energy estimate of $\me$. In order to overcome this difficulty, we consider a perturbed energy function $\me^{\sigma_1,\sigma_2}$ defined in \eqref{f_per} by employing a type of Bogovskii's estimate in the periodic domain(see Lemma \ref{lem_bogo}). Together with a careful analysis using several technical lemmas presented in Section \ref{sec:pre}, this makes it possible to obtain the desired estimate of large-time behavior of solutions.

Unfortunately, our strategy for the estimate of large-time behavior of classical solutions can not be applied to the whole space case since we use the Poincar\'e inequality to get proper dissipation rates from the drag force(see Lemma \ref{lem_gg}). To the best of author's knowledge, the large-time behavior for the types of Vlasov/Naiver-Stokes or Euler/Navier-Stokes equations in the whole space is still an open issue. 

\begin{remark}1. The global solution obtained in Theorem \ref{thm_ext} satisfies the assumptions in Theorem \ref{thm_large}. 

2. The estimate of large-time behavior can be essentially used for global well-posedness of solutions for types of two-phase fluid models. For instance, in \cite{CK, HKK}, the pressureless Euler/Navier-Stokes equations are considered, and the {\it a priori} large-time behavior estimate together with the bootstrapping argument played an important role in constructing the global classical solutions in time.

3. In the perturbation framework employed in \eqref{main_eq}, the Lyapunov functional $\ml$ in \eqref{Lyap} will be replaced by
\[
\ml_p := \int_{\T^3} e^h |u-\tilde m_c|^2 dx + \int_{\T^3} (e^h - 1)^2\,dx + \int_{\T^3} (n+1)| v - \tilde j_c|^2 dx + |\tilde m_c - \tilde j_c|^2 + \int_{\T^3} n^2 dx,
\]
with 
\[
\tilde m_c(t) := \int_{\T^3} e^h u\, dx \quad \mbox{and} \quad \tilde j_c(t) := \int_{\T^3} (n+1) v \,dx,
\]
due to $\intt n_0\,dx = 0$ and $\intt e^{h_0} \,dx = 1$. Then, under the assumptions on the initial data \eqref{ass_ini}, we can easily find that the conditions in Theorem \ref{thm_large} are verified, and this yields that the Lyapunov functional $\ml_p$ satisfies
\[
\ml_p(t) \ls \ml_p(0)e^{-Ct}, \quad t \geq 0.
\]
Furthermore, we can obtain the convergences of $u$ and $v$ to $(\tilde m_c(0) + \tilde j_c(0))/2$ in $L^\infty(\T^3)$ as time goes to infinity exponentially fast. This shows the alignment between two fluid velocities. More precisely, we use the Gagliardo-Nirenberg interpolation inequality to get
\[
\|u - \tilde m_c\|_{L^\infty(\T^3)} \ls \|u - \tilde m_c\|_{H^{s'}(\T^3)} \ls \|u - \tilde m_c\|_{L^2(\T^3)}^{1 - \beta}\|u - \tilde m_c\|_{H^{s' + 1}(\T^3)}^{\beta} \ls e^{- \tilde C t},
\]
for some $\tilde C > 0$, where $\beta = \frac{s'}{s'+1}$ with $s' > 3/2$ and we used 
\[
\intt |u - \tilde m_c|^2\,dx = \intt \frac{e^h}{e^h}|u - \tilde m_c|^2\,dx \leq C\intt e^h | u -\tilde m_c|^2 \,dx \ls \ml_p(0) e^{-Ct},
\]
and $u - \tilde m_c \in H^{s'+1}(\T^3)$ due to $\|h\|_{L^\infty(\T^3)} \ls \epsilon_0 \ll 1$ and $s'+1 > 5/2$. Similarly, we deduce
\[
\|v - \tilde j_c\|_{L^\infty(\T^3)} \ls e^{-\tilde Ct} \quad \mbox{for some} \quad \tilde C > 0.
\]
On the other hand, it follows from the conservation of total momentum(see Lemma \ref{lem_energy}) that
\[
\tilde m_c(t) + \tilde j_c(t) = \tilde m_c(0) + \tilde j_c(0),
\]
and this implies
\[
|\tilde m_c(t) - \tilde j(t)| = 2\lt|\tilde m_c(t) - \frac12\lt(\tilde m_c(0) + \tilde j_c(0) \rt) \rt| = 2\lt|\tilde j_c(t) - \frac12\lt(\tilde m_c(0) + \tilde j_c(0) \rt) \rt|
\]
Combining the all of the above ingredients, we have
$$\begin{aligned}
&\lt\|u - \frac12\lt( \tilde m_c(0) + \tilde j_c(0) \rt)\rt\|_{L^\infty(\T^3)} + \lt\|v - \frac12\lt( \tilde m_c(0) + \tilde j_c(0) \rt)\rt\|_{L^\infty(\T^3)}\cr
&\ \leq \| u - \tilde m_c(t)\|_{L^\infty(\T^3)} + \lt| \tilde m_c(t) - \frac12\lt( \tilde m_c(0) + \tilde j_c(0) \rt)\rt| + \| v- \tilde j_c\|_{L^\infty(\T^3)} + \lt|\tilde j_c(t) - \frac12\lt( \tilde m_c(0) + \tilde j_c(0) \rt) \rt|\cr
&\ = \|u - \tilde m_c\|_{L^\infty(\T^3)} + |\tilde m_c - \tilde j_c | + \|v - \tilde j_c\|_{L^\infty(\T^3)}\cr
&\ \ls e^{-\tilde C t} \quad \mbox{for some} \quad \tilde C > 0.
\end{aligned}$$
\end{remark}

\subsection{Organization of the paper} In Section \ref{sec:pre}, we provide several useful estimates and {\it a priori} energy estimates that will play an important role later. Section \ref{sec:ap} is devoted to show the local and global existence of classical solutions and concludes the proof of Theorem \ref{thm_ext}. For the local existence and uniqueness of classical solutions to the reformulated system \eqref{main_eq}, the standard arguments developed for the types of conservation laws can be applied. Global existence of classical solutions is then obtained by the {\it a priori} estimates of solutions with the aid of the careful analysis of the drag forcing term. Finally, in Section \ref{sec:lg}, we give a detailed proof of the large-time behavior of global classical solutions to the original system \eqref{f_eq} showing the exponential alignment between two fluid velocities and convergences of densities to its averaged quantities as presented in Theorem \ref{thm_large}. \newline

Before closing the section, we introduce several notations used throughout the paper. For a function $f(x)$, $\|f\|_{L^p}$ denotes the usual $L^p(\T^3)$-norm. We also denote by $C$ a generic positive constant depending only on the norms of the data, but independent of $T$. For simplicity, we often drop $x$-dependence of differential operators $\nabla_x$, that is, $\nabla f := \nabla_x f$ and $\Delta f := \Delta_x f$. For any nonnegative integer $s$, $H^s$ denotes the $s$-th order $L^2$ Sobolev space. $\mc^s([0,T];E)$ is the set of $s$-times continuously differentiable functions from an interval $[0,T]\subset \R$ into a Banach space $E$, and $L^p(0,T;E)$ is the set of the $L^2$ functions from an interval $(0,T)$ to a Banach space $E$. $\nabla^s$ denotes any partial derivative $\pa^\alpha$ with multi-index $\alpha, |\alpha| = s$. 

%
%
%
%
\section{Preliminaries}\label{sec:pre}
In this section, we provide several useful Sobolev inequalities and energy estimates for the system \eqref{f_eq}. All of these technical estimates will be significantly used later for the {\it a priori} estimates of solutions and the large-time behavior.

We first recall Moser-type inequality.
\begin{lemma}\label{lem:bern} (i) For any pair of functions $f,g \in (H^k \cap L^\infty)(\om)$, we obtain
\[
\|\nabla^k(fg)\|_{L^2} \ls \|f\|_{L^\infty}\|\nabla^k g\|_{L^2} + \|\nabla^k f\|_{L^2}\|g\|_{L^\infty}.
\]
Furthermore if $\nabla f \in L^\infty(\om)$ we have
\[
\|\nabla^k(fg) - f\nabla^k g\|_{L^2} \ls \|\nabla f\|_{L^\infty}\|\nabla^{k-1}g\|_{L^2} + \|\nabla^k f\|_{L^2}\|g\|_{L^\infty}.
\]
(ii) Let $k \in \N, p \in [1,\infty], f \in \mc^k(\om)$. Then there exists a positive constant $c = c(k,p,h)$ such that 
\[
\|\nabla^k f(w)\|_{L^p} \leq c\|w\|_{L^\infty}^{k-1}\|\nabla^k w\|_{L^p},
\]
for all $w \in (W^{k,p} \cap L^\infty)(\om)$.
\end{lemma}

In the lemmas below, we present useful estimates which give some information for the evolution of the densities $\rho$ and $n$. For the details of the proof, we refer to \cite{Choi1}.
\begin{lemma}\label{lem_press}
1. Let $r_0, \bar{r} > 0$ and $\gamma \geq 1$ be
given constants, and set
\[
f(\gamma,r;r_0) := r\int_{r_0}^{r} \frac{s^{\gamma} - r_0^{\gamma}}{s^2} \,ds,
\]
for $r \in [0,\bar{r}]$. Then, there exists a positive constant
$C>0$ such that
\[
\frac{1}{C(\gamma,r_0, \bar{r})}(r - r_0)^2 \leq f(\gamma,r;r_0) \leq C(\gamma,r_0, \bar{r})(r - r_0)^2 \quad \mbox{for all } r \in [0,\bar{r}].
\]
2. There hold
\[
\frac{d}{dt}\int_\om \rho \ln \rho\,dx = \frac{d}{dt}\int_\om f(1,\rho\,;1)dx
\]
and
\[
\frac{1}{\gamma-1}\frac{d}{dt}\int_{\om} (n+1)^\gamma dx = \frac{d}{dt}\int_{\om} f(\gamma,n+1;1) dx.
\]
\end{lemma}
\begin{lemma}\label{lem_eqv}For $0 < a \leq f(x) \leq b$ with $a \leq 1 \leq b$, there exist positive constants $C(a)>0$ and $C(b) > 0$ such that
\[
C(b) \int_{\om} \lt(f - 1\rt)^2 dx \leq \int_{\om} \lt(\ln f\rt)^2 dx \leq C(a)\int_{\om} \lt(f - 1\rt)^2 dx,
\]
where $C(a)$ and $C(b)$ are given by
\[
C(a):= \max\lt\{1, \lt( \frac{\ln a}{1 - a}\rt)^2 \rt\} \quad \mbox{and} \quad C(b):= \min\lt\{1, \lt( \frac{\ln b}{b - 1}\rt)^2 \rt\},
\]
respectively.
\end{lemma}
We next show {\it a priori} energy estimates for the system \eqref{main_eq}.
\begin{lemma}\label{lem_energy} Let $(\rho,u,n,v)$ be any global classical solutions to  \eqref{main_eq}-\eqref{ini_main_eq}.
Then we have
\begin{align*}
\begin{aligned}
&(i) \,\text{Conservations of the masses and total momentum:}\cr
& \qquad \qquad \frac{d}{dt}\int_{\om} \rho \,dx = \frac{d}{dt}\int_{\om} n \,dx = 0 \quad \mbox{and} \quad \frac{d}{dt}\int_{\om} (\rho u + (n+1)v) \,dx = 0.\cr
&(ii)\,\text{Dissipation of the total energy:}\cr
&\qquad  \qquad \frac12\frac{d}{dt}E(t)+ \mu\int_{\om} |\nabla v|^2 dx + (\mu + \lambda)\int_{\om} |\nabla \cdot v|^2 dx +\int_{\om} \rho |u-v|^2 dx = 0.
\cr
&\quad \quad \text{where}\cr
&\qquad \qquad E(t) := \int_{\om} \rho|u|^2 dx +2 \int_\om \rho \ln \rho\,dx + \int_\om (n+1)|v|^2 dx + \frac{2}{\gamma-1}\int_{\om} (n+1)^\gamma dx.
\end{aligned}
\end{align*}
\end{lemma}
\begin{proof}Straightforward computations yield the estimates $(i)$. For the energy estimate $(ii)$, we use 
\[
\int_\om u \cdot \nabla \rho\,dx = \frac{d}{dt}\int_\om \rho \ln \rho\,dx \quad \mbox{and} \quad \int_\om v \cdot \nabla p(n+1)\,dx = \frac{1}{\gamma-1}\frac{d}{dt}\int_{\om} (n+1)^\gamma dx.
\]
This completes the proof.
\end{proof}

We notice from Lemma \ref{lem_energy} that the masses $\rho_c(t) = \int_\om \rho\,dx$ and $n_c(t) = \int_\om n\,dx$ are conserved in time, i.e., $\rho_c(t) = \rho_c(0)$ and $n_c(t) = n_c(0)$ for $t \geq 0$. Thus, for the sake of simplicity, we denote $\rho_c(0)$ and $n_c(0)$ by $\rho_c$ and $n_c$, respectively.

\begin{remark}\label{rmk:energy}1. Let $(\rho, u, n, v)$ be any global classical solutions to the system \eqref{main_eq}-\eqref{ini_main_eq}. Suppose $\rho \in [0,\bar\rho]$ and $n + 1 \in [0, \bar n + 1]$. Then it follows from Lemmas \ref{lem_press} and \ref{lem_energy} that

$$\begin{aligned}
&\frac12\frac{d}{dt}\lt(\int_\om \rho|u|^2 dx + 2\int_\om \rho \int_1^\rho \frac{s - 1}{s^2}ds dx + \int_\om (n+1)|v|^2 dx + 2\int_\om (n+1)\int_1^{n+1}\frac{s^\gamma - 1}{s^2}dsdx \rt)\cr
&\qquad + \mu\int_{\om} |\nabla v|^2 dx + (\mu + \lambda)\int_{\om} |\nabla \cdot v|^2 dx +\int_{\om} \rho |u-v|^2 dx = 0.
\end{aligned}$$
Then we deduce from Lemma \ref{lem_press} that

\begin{align*}
\begin{aligned}
&\int_\om \rho|u|^2 dx + \int_\om (\rho - 1)^2 dx + \int_\om n^2\,dx+ \int_\om(n+1)|v|^2 dx \cr
&\quad + \mu\int_0^t\int_\om |\nabla v|^2 dx ds + (\mu + \lambda)\int_0^t\int_\om |\nabla \cdot v|^2 dx ds+ \int_0^t \int_\om \rho|u-v|^2 dx ds\cr
&\quad \qquad \ls \int_\om \rho_0|u_0|^2 dx + \int_\om (\rho_0 - 1)^2 dx +\int_\om n_0^2\,dx + \int_\om (n_0 + 1)|v_0|^2 dx,
\end{aligned}
\end{align*}
and furthermore we also find 
\begin{align*}
\begin{aligned}
&\int_\om e^h|u|^2 dx + \int_\om h^2 dx + \int_\om n^2\,dx+ \int_\om(n+1)|v|^2 dx \cr
&\quad + \mu\int_0^t\int_\om |\nabla v|^2 dx ds + (\mu + \lambda)\int_0^t\int_\om |\nabla \cdot v|^2 dx ds+ \int_0^t \int_\om e^h|u-v|^2 dx ds\cr
&\quad \qquad \ls \int_\om e^{h_0}|u_0|^2 dx + \int_\om h_0^2 \,dx +\int_\om n_0^2\,dx + \int_\om (n_0 + 1)|v_0|^2 dx,
\end{aligned}
\end{align*}
for $h = \ln \rho \in L^\infty(\om \times \R_+)$, due to Lemma \ref{lem_eqv}. 

2. From Lemma \ref{lem_energy}, we obtain
\[
2\mu\int_0^\infty \int_\om |\nabla v|^2\,dx ds + 2\int_0^\infty \int_\om \rho|u-v|^2\,dx ds \leq E_0.
\]
Thus we asymptotically have
\[
\lim_{t\to\infty} \int_{t}^{t+1} \int_\om |\nabla v|^2\,dx ds = \lim_{t\to\infty}\int_t^{t+1} \int_\om \rho|u-v|^2\,dx ds = 0.
\]
\end{remark}
%
%
%
%
\section{Global existence and uniqueness of classical solutions}\label{sec:ap}
%
%
%
%
\subsection{Local existence}
In the theorem below, we provide local existence and uniqueness of strong solutions to the system \eqref{main_eq}. Since local existence theories for each equation have been well developed in $H^s$ Sobolev space as we mentioned in Section \ref{sec:intro}, we omit its detailed proof. For the readers who are interested in it, we refer to \cite{Majda} and references therein.

\begin{theorem}\label{thm:local-ext} Let $s > 5/2$, and suppose  $(h_0,u_0,n_0,v_0) \in H^{s}(\om) \times H^{s}(\om) \times H^{s}(\om) \times H^{s}(\om)$. Then, for any positive constants $\epsilon_0 < M_0$, there exists a positive constant $T_0$ depending only on $\epsilon_0$ and $M_0$ such that if $\|(h_0,u_0,n_0,v_0)\|_{H^s} \leq \epsilon_0$, then the system \eqref{main_eq}-\eqref{ini_main_eq} admits a unique solution $(h,u,n,v) \in \mathcal{I}_s(T_0)$ satisfying 
\[
\sup_{0 \leq t \leq T_0} \|(h,u,n,v)\|_{H^s} \leq M_0.
\]
\end{theorem}

We next provide the equivalence relation between the classical solutions to the system \eqref{f_eq} and \eqref{main_eq}.

\begin{proposition}\label{prop:equiv} For any fixed $T>0$, if $(\rho,u,n,v)\in \mc^2(\om \times [0,T])$ solves the system \eqref{f_eq}-\eqref{ini_f_eq} with $\rho > 0$ and $n > 0$, then $(h,u,n,v) \in \mc^2(\om \times [0,T])$ solves the system \eqref{main_eq}-\eqref{ini_main_eq} with $e^h > 0$ and $n+1 > 0$. Conversely, if $(h,u,n,v) \in \mc^2(\om \times [0,T])$ solves the system \eqref{main_eq}-\eqref{ini_main_eq} with $e^h >0$ and $n+1 > 0$, then $(\rho,u,n,v) \in \mc^2(\om \times [0,T])$ solves the system \eqref{f_eq}-\eqref{ini_f_eq} with $\rho > 0$ and $n > 0$.
\end{proposition}
\begin{proof}
The proof is straightforward, and the positivity of the densities is obtained from the corresponding positivity of the initial densities by using the method of characteristics. 
\end{proof}

%
%
%
%
\subsection{Global existence} In this part, we present the {\it a priori} estimates for the global existence of the classical solutions to the system \eqref{main_eq}. For this, we define
\[
\mh(T;s) := \sup_{0 \leq t \leq T}\|(h(t),u(t),n(t),v(t))\|_{H^s}^2 \quad \mbox{and} \quad \mh_0(s):= \|(h_0,u_0,n_0,v_0)\|_{H^s}^2.
\]
We first provide a uniform bound estimate of a zeroth-order of $\mh(T;s)$ in time.
\begin{lemma}\label{lem_sl2} 
Let $s>5/2$ and $T>0$ be given. Suppose that $\mh(T;s) \leq \epsilon_1$ for sufficiently small $\epsilon_1>0$. Then we obtain
\[
\mh(T;0) \leq C\mh_0(0),
\]
where $C>0$ is independent of $T$. 
\end{lemma}
\begin{proof}We choose $\epsilon_1>0$ small enough so that 
\[
\sup_{0 \leq t \leq T}\|e^{h(\cdot,t)} - 1\|_{L^\infty} \leq \frac12 \quad \mbox{and} \quad \sup_{0\leq t \leq T}\|n(\cdot,t)\|_{L^\infty} \leq \frac12.
\]
Then it follows from Remark \ref{rmk:energy} that
\begin{align*}
\begin{aligned}
&\int_\om |u|^2 dx + \int_\om h^2 dx + \int_\om n^2\,dx+ \int_\om |v|^2 dx \cr
&\quad + 2\mu\int_0^t\int_\om |\nabla v|^2 dx ds + 2(\mu + \lambda)\int_0^t\int_\om |\nabla \cdot v|^2 dx ds+ 2\int_0^t \int_\om e^h|u-v|^2 dx ds\cr
&\quad \qquad \ls \int_\om |u_0|^2 dx + \int_\om h_0^2 \,dx +\int_\om n_0^2\,dx + \int_\om |v_0|^2 dx.
\end{aligned}
\end{align*}
This concludes the desired result.
\end{proof}
For a higher order derivative, we first provide the estimate of $\|\nabla^k(h,u)\|_{L^2}$ for $1 \leq k \leq s$.
\begin{lemma}\label{lem_hu} 
Let $s>5/2$ and $T>0$ be given. Suppose that $\mh(T;s) \leq \epsilon_1$ for sufficiently small $\epsilon_1>0$. Then we obtain
\bq\label{est_hu}
\frac{d}{dt}\|\nabla^k (h,u)\|_{L^2}^2 + \frac32\|\nabla^k u\|_{L^2}^2 \leq C\epsilon_1\|\nabla^k (h,u)\|_{L^2}^2 + 2\|\nabla^k v\|_{L^2}^2 \quad \mbox{for} \quad 1 \leq k \leq s,
\eq
where $C>0$ is independent of $T$. 
\end{lemma}
\begin{proof}
For $1 \leq k \leq s+1$, it follows from \eqref{main_eq} that
\begin{align*}
\begin{aligned}
&\frac12\frac{d}{dt}\int_{\om} |\nabla^k h|^2 + |\nabla^k u|^2 dx \cr
&\quad = - \int_{\om} \nabla^k h \cdot \lt( \nabla^k(\nabla h \cdot u) + \nabla^k(\nabla \cdot u) \rt) dx - \int_{\om} \nabla^k u \cdot \lt( \nabla^k(u \cdot \nabla u) + \nabla^{k+1}h \rt) dx\cr
&\qquad - \int_{\om} \nabla^k u \cdot \lt(\nabla^k u - \nabla^k v \rt) dx\cr
&\quad = \frac12\int_{\om} (\nabla \cdot u)(|\nabla^k h|^2 + |\nabla^k u|^2)dx - \int_{\om} \nabla^k h \cdot [\nabla^k, u \cdot \nabla ]h\, dx\cr
&\qquad - \int_{\om} \nabla^k u \cdot [\nabla^k,u\cdot \nabla]u \,dx - \int_{\om} \nabla^k u \cdot \lt(\nabla^k u - \nabla^k v \rt) dx\cr
&\quad =: \sum_{i=1}^4 I_i,
\end{aligned}
\end{align*}
where $[\cdot,\cdot]$ denotes the commutator operator, i.e., $[A,B] = AB - BA$. Then we now estimate $I_i,i=1,\cdots,4$ as follows.
\begin{align*}
\begin{aligned}
I_1 &\leq \frac12\|\nabla u\|_{L^\infty}(\|\nabla^k h\|_{L^2}^2+ \|\nabla^k u\|_{L^2}^2) \ls \|\nabla (h,u)\|_{L^\infty}\|\nabla^k (h,u)\|_{L^2}^2 \leq C\epsilon_1\|\nabla^k (h,u)\|_{L^2}^2,\cr
I_2 &\ls \|\nabla^k h\|_{L^2}(\|\nabla^k u\|_{L^2}\|\nabla h\|_{L^\infty} + \|\nabla^k h\|_{L^2}\|\nabla u\|_{L^\infty} )\ls \|\nabla (h,u)\|_{L^\infty}\|\nabla^k (h,u)\|_{L^2}^2\cr
&\leq C\epsilon_1\|\nabla^k (h,u)\|_{L^2}^2,\cr
I_3 & \ls \|\nabla^k u\|_{L^2}^2\|\nabla u\|_{L^\infty}\ls \|\nabla (h,u)\|_{L^\infty}\|\nabla^k (h,u)\|_{L^2}^2\leq C\epsilon_1\|\nabla^k (h,u)\|_{L^2}^2,\cr
I_4 &\leq - \frac34\|\nabla^k u\|_{L^2}^2 + \|\nabla^k v\|_{L^2}^2,
\end{aligned}
\end{align*}
due to Lemma \ref{lem:bern}. This yields 
\[
\frac{d}{dt}\|\nabla^k (h,u)\|_{L^2}^2 + \frac32\|\nabla^k u\|_{L^2}^2 \leq C\epsilon_1\|\nabla^k (h,u)\|_{L^2}^2 + 2\|\nabla^k v\|_{L^2}^2,
\]
where $C>0$ is independent of $T$. 
\end{proof}
\begin{remark}\label{rmk_u}In a similar fashion as in Lemma \ref{lem_hu}, we also deduce that for $1 \leq k \leq s$
$$\begin{aligned}
\frac12\frac{d}{dt}\int_\om |\nabla^{k-1}u|^2\,dx &= \int_\om \nabla^{k-1}u \cdot \nabla^{k-1}u_t\,dx\cr
&\leq \int_\om \nabla^{k-1}u\cdot\lt( - \nabla^{k-1}(u \cdot \nabla u) - \nabla^k h - \nabla^{k-1}(u-v)\rt)dx\cr
&\leq C\epsilon_1 \|\nabla^{k-1}u\|_{L^2}^2 + \|\nabla^{k-1}u\|_{L^2}\|\nabla^k h\|_{L^2} - \frac12\|\nabla^{k-1}u\|_{L^2}^2 + \frac12\|\nabla^{k-1}v\|_{L^2}^2\cr
&\leq \frac14\|\nabla^k h\|_{L^2}^2 + C\|\nabla^{k-1}(u,v)\|_{L^2}^2,
\end{aligned}$$
where $C>0$ is independent of $T$. 
\end{remark}
We notice that we do not have the dissipation rate of $h$ in the estimate \eqref{est_hu}. In order to get it, the estimate of $\|h_t\|_{H^{s-1}}$ is used for the compressible Euler equation with damping \cite{PZ,STW} since the required dissipation rate of $h$ can be bounded from above by Sobolev norms of other solutions under the smallness assumption on the solutions. Indeed, it follows from the part of Euler equations $\eqref{main_eq}_1$-$\eqref{main_eq}_2$ that
$$\begin{aligned}
h_t &= -\nabla h \cdot u  - \nabla \cdot u,\cr
\nabla h &= - u_t - u \cdot \nabla u - (u - v).
\end{aligned}$$
Then, for $0 \leq k \leq s-1$, we get
$$\begin{aligned}
\|\nabla^k h_t \|_{L^2} &\leq \|\nabla^k(\nabla h \cdot u)\|_{L^2} + \|\nabla^{k+1} u\|_{L^2}\cr
&\ls \|\nabla h\|_{H^{s-1}}\|u\|_{L^\infty} + \|\nabla h\|_{L^\infty} \|u\|_{H^{s-1}} + \|\nabla u\|_{H^{s-1}}\cr
&\ls \|u\|_{H^s}.
\end{aligned}$$
Similarly, we find 
\[
\|\nabla h\|_{H^{s-1}} \ls \|u_t\|_{H^{s-1}} + \|u\|_{H^{s-1}} + \|v\|_{H^{s-1}}.
\]
Thus we obtain
\bq\label{s_1}
\|h_t\|_{H^{s-1}} + \|\nabla h\|_{H^{s-1}} \ls \|u\|_{H^s} + \|u_t\|_{H^{s-1}} + \|v\|_{H^{s-1}}.
\eq
When there is no interactions with fluids, i.e., $v = 0$ in \eqref{s_1} and the system $\eqref{main_eq}_1$-$\eqref{main_eq}_2$ reduces to the Euler equations with linear damping, the above observation enables us to have the uniform bound estimates of solutions. This strategy is taken into consideration for the isothermal Euler/incompressible Navier-Stokes equations in \cite{Choi1}. However, for our system \eqref{main_eq}, it will be very complicated if we estimate $\|(h_t,u_t,n_t,v_t)\|_{H^{s-1}}$ to get the dissipation rate for the densities $h$ and $n$. To this end, we give the following lemma inspired by \cite{MN1, MN2} which gives the dissipation rate for the density $h$.
\begin{lemma}\label{lem_h}
Let $s>5/2$ and $T>0$ be given. Suppose that $\mh(T;s) \leq \epsilon_1$ for sufficiently small $\epsilon_1>0$. Then there exists a $C > 0$ independent of $T$ such that 
$$
\begin{aligned}
&\frac{d}{dt}\int_\om \nabla^{k-1}u \cdot \nabla^k h\,dx + \frac34\|\nabla^k h\|_{L^2}^2  \cr
&\quad \leq C\epsilon_1\|\nabla^k(h,u)\|_{L^2}^2 + C\epsilon_1\|\nabla^{k-1}u\|_{L^2}^2 + \|\nabla^k u\|_{L^2}^2+ C\lt( \|\nabla^{k-1}u\|_{L^2}^2 + \|\nabla^{k-1}v\|_{L^2}^2\rt),
\end{aligned}
$$
for $1 \leq k \leq s$. 
\end{lemma}
\begin{proof}For $1 \leq k \leq s$, it is a straightforward to get
$$\begin{aligned}
\frac{d}{dt}\int_\om \nabla^{k-1}u \cdot \nabla^k h\,dx &=\int_\om \nabla^{k-1} u_t \cdot \nabla^k h\,dx + \int_\om \nabla^{k-1}u \cdot \nabla^k h_t\,dx\cr
&= J_1 + J_2,
\end{aligned}$$
where $J_1$ is estimated as follows.
$$\begin{aligned}
J_1 &= -\int_\om \nabla^k h \cdot \nabla^{k-1}\lt( u \cdot \nabla u + \nabla h + (u-v)\rt)dx\cr
&\leq \|\nabla^k h\|_{L^2}\lt( \|\nabla^{k-1} u\|_{L^2}\|\nabla u\|_{L^\infty} + \|u\|_{L^\infty}\|\nabla^k u\|_{L^2}\rt) - \|\nabla^k h\|_{L^2}^2\cr
&\quad + \|\nabla^k h\|_{L^2}\lt(\|\nabla^{k-1}u\|_{L^2} + \|\nabla^{k-1}v\|_{L^2} \rt)\cr
&\leq C\epsilon_1\|\nabla^k(h,u)\|_{L^2}^2 + C\epsilon_1\|\nabla^{k-1}u\|_{L^2}^2 - \frac34 \|\nabla^k h\|_{L^2}^2 + C\lt( \|\nabla^{k-1}u\|_{L^2}^2 + \|\nabla^{k-1}v\|_{L^2}^2\rt).
\end{aligned}$$
For the estimate of $J_2$, we use the integration by parts together with the continuity equation $\eqref{main_eq}_1$ to find
$$\begin{aligned}
J_2 &= -\int_\om \nabla^k u \cdot \nabla^{k-1}h_t\,dx \cr
&= \int_\om \nabla^k u \cdot \lt( \nabla^{k-1}(\nabla h \cdot u + \nabla \cdot u)\rt)dx\cr
&\leq \|\nabla^k u\|_{L^2}\lt( \|\nabla h\|_{L^\infty}\|\nabla^{k-1}u\|_{L^2} + \|\nabla^k h\|_{L^2}\|u\|_{L^\infty}\rt) +\|\nabla^k u\|_{L^2}^2 \cr
&\leq C\epsilon_1\|\nabla^k(h,u)\|_{L^2}^2 + C\epsilon_1\|\nabla^{k-1}u\|_{L^2}^2 + \|\nabla^k u\|_{L^2}^2.
\end{aligned}$$
Thus we have
$$\begin{aligned}
&\frac{d}{dt}\int_\om \nabla^{k-1}u \cdot \nabla^k h\,dx + \frac34\|\nabla^k h\|_{L^2}^2  \cr
&\quad \leq C\epsilon_1\|\nabla^k(h,u)\|_{L^2}^2 + C\epsilon_1\|\nabla^{k-1}u\|_{L^2}^2 + \|\nabla^k u\|_{L^2}^2+ C\lt( \|\nabla^{k-1}u\|_{L^2}^2 + \|\nabla^{k-1}v\|_{L^2}^2\rt).
\end{aligned}$$
This completes the proof.
\end{proof}
\begin{remark}\label{rmk_eqv}Note that
\[
\int_\om |\nabla^k (h,u)|^2\,dx + \int_\om \nabla^{k-1}u \cdot \nabla^k h \,dx+ \int_\om |\nabla^{k-1} u|^2\,dx \approx \|\nabla^k h\|_{L^2}^2 + \|\nabla^{k-1}u\|_{H^1}^2,
\]
i.e., there exists a positive constant $C$ such that
$$\begin{aligned}
&\frac1C\lt( \|\nabla^k h\|_{L^2}^2 + \|\nabla^{k-1}u\|_{H^1}^2\rt)\cr
& \leq\int_\om |\nabla^k (h,u)|^2 \,dx + \int_\om \nabla^{k-1}u \cdot \nabla^k h \,dx+ \int_\om |\nabla^{k-1} u|^2\,dx \leq C\lt( \|\nabla^k h\|_{L^2}^2 + \|\nabla^{k-1}u\|_{H^1}^2\rt).
\end{aligned}$$
\end{remark}
We next give the estimate of $\|\nabla^k v\|_{L^2}$ for $1 \leq k \leq s$. The similar framework is employed in \cite{CK}. However, it can not be directly applied here, since the periodic domain is considered in \cite{CK}, and in addition Poincar\'e inequality is also used for the estimate.
\begin{lemma}\label{lem_v}
Let $s>5/2$ and $T>0$ be given. Suppose that $\mh(T;s) \leq \epsilon_1$ for sufficiently small $\epsilon_1>0$. Then we have
$$\begin{aligned}
&\frac{d}{dt}\|\nabla^k v\|_{L^2}^2 + \frac{1}{3}\|\nabla^k v\|_{L^2}^2 + C_1\mu\|\nabla^{k+1}v\|_{L^2}^2\cr
&\quad \leq C\epsilon_1\|\nabla^k (h,n,v)\|_{L^2}^2 +3\|\nabla^k u\|_{L^2}^2 +C\|\nabla n\|_{H^{k-1}}^2 \quad \mbox{for} \quad 1 \leq k \leq s,
\end{aligned}
$$
for some positive constant $C_1 > 0$. Here $C>0$ is independent of $T$. 
\end{lemma}
\begin{proof}
It follows from the Navier-Stokes equations $\eqref{main_eq}_4$ that 
\[
v_t + v \cdot \nabla v + \frac{\nabla p(n+1)}{n+1} + \frac{Lv}{n+1} = \frac{e^h}{n+1}(u-v).
\]
Then, for $1 \leq k \leq s$, by applying $\nabla^k$ to the above equation we find
\[
\nabla^{k} v_t + \nabla^{k}(v\cdot \nabla v) + \nabla^{k}\lt(\frac{\nabla p(n+1)}{n+1}\rt) + \nabla^{k}\lt(\frac{Lv}{n+1}\rt) = \nabla^{k}\lt(\frac{e^h}{n+1}(u-v)\rt),
\]
and it can be rewritten as 
\begin{align*}
\begin{aligned}
\nabla^{k} v_t &= - v \cdot \nabla^{k+1} v - [\nabla^{k},v\cdot \nabla]v + \nabla^{k}\lt(\frac{e^h}{n+1}(u-v)\rt)  - \nabla^{k} \lt(\frac{\nabla p(n+1)}{n+1}\rt) - \nabla^{k} \lt(\frac{Lv}{n+1}\rt).
\end{aligned}
\end{align*}
This yields that
\begin{align*}
\begin{aligned}
\frac12\frac{d}{dt}\|\nabla^{k} v\|_{L^2}^2 
&= -\int_\om (v\cdot \nabla^{k+1}v) \cdot \nabla^{k} v +  [\nabla^{k},v\cdot \nabla]v \cdot \nabla^{k} v - \nabla^{k}\lt(\frac{e^h}{n+1}(u-v)\rt) \cdot \nabla^{k} v \,dx\cr
&-\int_\om \nabla^{k} \lt(\frac{\nabla p(n+1)}{n+1}\rt) \cdot \nabla^{k} v + \nabla^{k} \lt(\frac{Lv}{n+1}\rt) \cdot \nabla^{k} v\,dx\cr
&=: \sum_{i=1}^5 K_i.
\end{aligned}
\end{align*}
We first easily estimate $K_1$ and $K_2$ as 
$$\begin{aligned}
K_1 &= \frac12\int_\om (\nabla \cdot v)|\nabla^k v|^2\,dx \leq C \epsilon_1 \|\nabla^k v\|_{L^2}^2,\cr
K_2 &\leq \|[\nabla^{k}, v \cdot \nabla]v\|_{L^2}\|\nabla^{k} v\|_{L^2} \leq C \epsilon_1\|\nabla^{k} v\|_{L^2}^2,
\end{aligned}$$
due to Lemma \ref{lem:bern}. 

For the $K_3$, we divide it into two parts:
$$\begin{aligned}
K_3 &= \int_\om \nabla^k v\cdot \frac{e^h}{n+1}\nabla^k(u-v)\,dx + \int_\om \nabla^k v \cdot\lt(\nabla^k \lt( \frac{e^h}{n+1}(u-v)\rt) - \frac{e^h}{n+1}\nabla^k (u-v) \rt)\cr
&=: K_3^1 + K_3^2.
\end{aligned}$$
By the smallness assumptions on the solutions, we can obtain $1/2 \leq e^{h(x,t)} \leq 3/2$ and $|n(x,t)|\leq 1/2$ for all $(x,t) \in \om \times [0,T]$, and this deduces
$$\begin{aligned}
K_3^1 &=\int_\om \frac{e^h}{n+1} \nabla^k u \cdot \nabla^k v\,dx - \int_\om \frac{e^h}{n+1} |\nabla^k v|^2\,dx\cr
&\leq \frac12\int_\om \frac{e^h}{n+1}|\nabla^k u|^2\,dx - \frac12\int_\om \frac{e^h}{n+1}|\nabla^k v|^2\,dx\cr
&\leq \frac32\int_\om |\nabla^k u|^2\,dx - \frac16\int_\om |\nabla^k v|^2\,dx,
\end{aligned}$$
due to $1/3 \leq e^h/(n+1) \leq 3$ for all $(x,t) \in \om \times [0,T]$.

We also obtain 
$$\begin{aligned}
K_3^2 &\leq \|\nabla^k v\|_{L^2}\lt(\lt\|\nabla \lt(\frac{e^h}{n+1}\rt) \rt\|_{L^\infty}\|\nabla^{k-1}(u-v)\|_{L^2} + \lt\|\nabla^k \lt(\frac{e^h}{n+1} \rt)\rt\|_{L^2} \|u - v\|_{L^\infty} \rt)\cr
&\leq \|\nabla^k v\|_{L^2}\lt( \|\nabla(h,n)\|_{L^\infty}\|\nabla^{k-1}(u-v)\|_{L^2} + \|\nabla^k (h,n)\|_{L^2}\|u - v\|_{L^\infty}\rt)\cr
&\leq C\epsilon_1 \|\nabla^k (h,n,v)\|_{L^2}^2  + C\epsilon_1\|\nabla^{k-1}(u,v)\|_{L^2}^2,
\end{aligned}$$
where we used the Sobolev inequality in Lemma \ref{lem:bern} to get
\[
\lt\|\nabla \lt(\frac{e^h}{n+1}\rt) \rt\|_{L^\infty} \ls \|\nabla(h,n)\|_{L^\infty} \quad \mbox{and} \quad \lt\|\nabla^k \lt(\frac{e^h}{n+1} \rt)\rt\|_{L^2} \ls \|\nabla^k(h,n)\|_{L^2}.
\]
This implies
$$\begin{aligned}
K_3 &\leq C\epsilon_1 \|\nabla^k (h,n,v)\|_{L^2}^2  + C\epsilon_1\|\nabla^{k-1}(u,v)\|_{L^2}^2 + \frac32\int_\om |\nabla^k u|^2\,dx - \frac16\int_\om |\nabla^k v|^2\,dx.
\end{aligned}$$
Integrating by parts, we get
$$\begin{aligned}
K_4 &= \int_{\om} \nabla^{k-1} \lt( \frac{\nabla p(n+1)}{n+1}\rt)\cdot \nabla^{k+1} v \,dx \cr
&=\int_{\om} \nabla^{k-1} \lt( \gamma(n+1)^{\gamma-2}\nabla n\rt)\cdot \nabla^{k+1} v \,dx \cr
&\leq C(1 - \delta_{k,1})\sum_{1 \leq l \leq k-1}\int_{\om} |\nabla^l ((n+1)^{\gamma-2})||\nabla^{k-l} n||\nabla^{k+1}v|\,dx + C\int_{\om} |\nabla^k n| |\nabla^{k+1} v|\,dx\cr
&\leq C(1 - \delta_{k,1})\sum_{1 \leq l \leq k-1}\|\nabla^l ((n+1)^{\gamma-2})\|_{H^1}\|\nabla^{k-l} n\|_{H^1}\|\nabla^{k+1}v\|_{L^2} + C\|\nabla^k n\|_{L^2}\|\nabla^{k+1}v\|_{L^2}\cr
&\leq C\|\nabla n\|_{H^{k-1}}\|\nabla^{k+1}v\|_{L^2}\cr
&\leq C\delta_1^{-1}\|\nabla n\|_{H^{k-1}}^2 + C\delta_1\|\nabla^{k+1}v\|_{L^2}^2,
\end{aligned}$$
due to the interpolation and Sobolev inequalities, where $\delta_1 >0$ will be determined later and $\delta_{i,j}$ denotes the Kronecker delta, i.e., $\delta_{i,j} = 1$ if $i=j$ and $\delta_{i,j} = 0$ if $i \neq j$.

We estimate $K_5$ by decomposing it into two parts as 
\begin{align*}
\begin{aligned}
K_5 &= \int_{\om} \nabla^{k} v \cdot \nabla^{k}\lt( \frac{\mu}{n+1}\nabla \cdot \nabla v + \frac{\mu + \lambda}{n+1}\nabla \nabla \cdot v \rt)dx\cr
&=:K_5^1 + K_5^2.
\end{aligned}
\end{align*}
Since the estimate of $K_5^2$ is very similar to that of $K_5^1$, we only provide the estimate of $K_5^1$. For this, we again rewrite it as the summation of three terms, $K_5^{1,i}, i=1,2,3$:

\begin{align*}
\begin{aligned}
K_5^1 &= \mu\int_{\om} \frac{1}{n+1}\nabla\cdot\nabla^{k+1}v \cdot \nabla^{k}v \,dx + \mu
\int_{\om}\nabla^{k}v \cdot \nabla^{k}\lt( \frac{1}{n+1}\rt) \nabla \cdot \nabla v \,dx\cr
&\quad + \mu(1 - \delta_{k,1})\sum_{1\leq l \leq k-1} \binom{k}{l}\int_{\om}\nabla^{k}v \cdot \nabla^l\lt( \frac{1}{n+1}\rt)\nabla^{k-l}\nabla \cdot \nabla v \,dx\cr
&=:K_{5}^{1,1} + K_{5}^{1,2} + K_{5}^{1,3}.
\end{aligned}
\end{align*}
Here $K_5^{1,i},i=1,2,3$ are estimated as follows.
\begin{align*}
K_{5}^{1,1}&= - \mu\int_{\om} \nabla\lt( \frac{1}{n+1}\nabla^{k}v\rt)\cdot \nabla^{k+1} v \,dx\cr
&= \mu\int_{\om} \lt( \frac{\nabla n}{(n+1)^2} \cdot \nabla^{k}v\rt)\cdot \nabla^{k+1} v \,dx - \mu\int_{\om} \frac{1}{n+1}|\nabla^{k+1} v|^2 dx\cr
&\leq C\epsilon_1\|\nabla^{k}v\|_{L^2}\|\nabla^{k+1}v\|_{L^2} - C{\mu}\|\nabla^{k+1} v\|_{L^2}^2\cr
&\leq C\epsilon_1\|\nabla^{k}v\|_{L^2}^2 - C(\mu - C\epsilon_1)\|\nabla^{k+1}v\|_{L^2}^2,\cr
K_{5}^{1,2}&\leq C\|\nabla^{k}v\|_{L^4}\|\nabla^2 v\|_{L^4}\lt\|\nabla^{k}\lt( \frac{1}{n+1}\rt) \rt\|_{L^2}\cr
&\leq C\|\nabla^{k}v\|_{H^1}\|\nabla^2 v\|_{H^1}\|\nabla^k n\|_{L^2}\cr
&\leq C\epsilon_1\|\nabla^{k+1}v\|_{L^2}^2 + C\epsilon_1\|\nabla^{k}v\|_{L^2}^2 + C\epsilon_1\|\nabla^{k} n\|_{L^2}^2,\cr
K_{5}^{1,3} &\leq C(1 - \delta_{k,1})\sum_{1\leq l \leq k-1}\|\nabla^{k}v\|_{L^4}\lt\| \nabla^l\lt( \frac{1}{n+1}\rt)\rt\|_{L^4}\|\nabla^{k+2-l} v\|_{L^2}\cr
&\leq C(1 - \delta_{k,1})\sum_{ 1 \leq l \leq k-1}\|\nabla^k v\|_{H^1}\|\nabla^l n\|_{H^1}\|\nabla^{k+2-l} v\|_{L^2}\cr
&\leq C(1 - \delta_{k,1})\|\nabla^{k}v\|_{H^1}\|\nabla n\|_{H^{k-1}}\|\nabla^3 v\|_{H^{k-2}} \cr
&\leq C\epsilon_1\|\nabla^{k+1}v\|_{L^2}^2 + C\epsilon_1\|\nabla^k v\|_{L^2}^2 + C\epsilon_1 \|\nabla n\|_{H^{k-1}}^2,
\end{align*}
where we used $\|f\|_{L^4} \ls \|f\|_{H^1}$ for $f \in H^1(\om)$ together with Lemma \ref{lem:bern}. This yields 
\[
K_{5}^1 \leq C\epsilon_1\|\nabla^k v\|_{L^2}^2 + C\epsilon_1 \|\nabla n\|_{H^{k-1}}^2- C(\mu - C\epsilon_1)\|\nabla^{k+1}v\|_{L^2}^2.
\]
Similarly, we obtain
\[
K_{5}^2 \leq C\epsilon_1\|\nabla^k v\|_{L^2}^2 + C\epsilon_1 \|\nabla n\|_{H^{k-1}}^2- C\epsilon_1\|\nabla^{k+1}v\|_{L^2}^2,
\]
and this implies
\[
K_{5} \leq C\epsilon_1\|\nabla^k v\|_{L^2}^2 + C\epsilon_1 \|\nabla n\|_{H^{k-1}}^2- C(\mu - C\epsilon_1)\|\nabla^{k+1}v\|_{L^2}^2.
\]
We now collect all of the above estimate to deduce
$$\begin{aligned}
&\frac{d}{dt}\|\nabla^k v\|_{L^2}^2 + \frac{1}{3}\|\nabla^k v\|_{L^2}^2 + C(\mu - C\epsilon_1 - C\delta_1)\|\nabla^{k+1}v\|_{L^2}^2\cr
&\quad \leq C\epsilon_1\|\nabla^k (h,n,v)\|_{L^2}^2 +3\|\nabla^k u\|_{L^2}^2 +C\delta_1^{-1}\|\nabla n\|_{H^{k-1}}^2 +C\epsilon_1 \|\nabla n\|_{H^{k-1}}^2.
\end{aligned}$$
By choosing $\delta_1 > 0$ small enough so that $C(\mu - C\epsilon_1 - C\delta_1) >0$, we conclude our desired result.
\end{proof}

We finally provide the estimate of $\|\nabla^k n\|_{L^2}$ for $1 \leq k \leq s$. For the same reason as before, we add a suitable cross term $\int_\om \frac{(n+1)^2}{2\mu + \lambda}\nabla^k n \cdot \nabla^{k-1} v\, dx$ to the usual $L^2$-norm estimate of $\nabla^k n$, and then get the desired dissipation rate for the density $n$. This proof is rather lengthy and technical, we postpone it to Appendix \ref{app_a}.

\begin{lemma}\label{lem_n}
Let $s>5/2$ and $T>0$ be given. Suppose that $\mh(T;s) \leq \epsilon_1$ for sufficiently small $\epsilon_1>0$. Then, for $1 \leq k \leq s$, we have
\begin{align*}
\begin{aligned}
&\frac{d}{dt}\int_{\om} \nabla^k n \cdot \lt( \frac{\nabla^k n}{2} + \frac{(n+1)^2}{2\mu + \lambda} \nabla^{k-1} v\rt) dx + C_2\|\nabla^k n\|_{L^2}^2\cr
&\quad \leq C\epsilon_1\|n\|_{H^{k-1}}^2 + C\epsilon_1\|v\|_{H^{k+1}}^2 + C\epsilon_1\|\nabla^{k-1}h\|_{L^2}^2 + C\|\nabla^k v\|_{L^2}^2 + C\|\nabla^{k-1}(u,v)\|_{L^2}^2,
\end{aligned}
\end{align*}
for some positive constant $C_2 > 0$. Here $C>0$ is independent of $T$. 
\end{lemma}
We now combine all of the estimates obtained in Lemmas \ref{lem_h}-\ref{lem_n} to have the uniform bound of $\mh(T;s)$ in time.
\begin{proposition}\label{prop_gd}Let $T > 0$ be given. Suppose $\mh(T;s) \leq \epsilon_1 \ll 1$. Then we have
\[
\mh(T;s) \leq  C_0\mh_0(s),
\]
where $C_0$ is a positive constant independent of $T$.
\end{proposition}
\begin{proof}We claim that there exists a positive constant $C$ independent of $T$ such that
\bq\label{est_ind}
\mh(T;k) \leq C\mh_0(k) \quad \mbox{for} \quad 0 \leq k \leq s.
\eq
For the proof of claim, we use the induction argument. First, it follows from Lemma \ref{lem_sl2} that the inequality \eqref{est_ind} holds for $k=0$. We now assume that \eqref{est_ind} holds for $k=m < s$. Then we deduce from Lemmas \ref{lem_hu}, \ref{lem_h}, and Remark \ref{rmk_u} that 
\begin{align}\label{est_hum}
\begin{aligned}
&\frac{d}{dt}\lt(\int_\om |\nabla^{m+1}(h,u)|^2\,dx + \int_\om \nabla^m u \cdot \nabla^{m+1} h \,dx + \int_\om |\nabla^m u|^2\,dx\rt)\cr
&\quad \leq - \frac12\|\nabla^{m+1}u\|_{L^2}^2 - \frac14\|\nabla^{m+1}h\|_{L^2}^2 + C\epsilon_1\|\nabla^{m+1}(h,u)\|_{L^2}^2 + 2\|\nabla^{m+1}v\|_{L^2}^2\cr
&\qquad + C\epsilon_1\|\nabla^m u\|_{L^2}^2 + C\|\nabla^m(u,v)\|_{L^2}^2\cr
&\quad \leq -\lt(\frac14 - C\epsilon_1\rt)\|\nabla^{m+1}(h,u)\|_{L^2}^2 + 2\|\nabla^{m+1}v\|_{L^2}^2 + C\mh_0(m),
\end{aligned}
\end{align}
where we also used the fact from the induction hypothesis that
\[
C\epsilon_1\|\nabla^m u\|_{L^2}^2 + C\|\nabla^m(u,v)\|_{L^2}^2 \leq C\|\nabla^m(u,v)\|_{L^2}^2 \leq C\mh(T;m) \leq C\mh_0(m).
\]
In order to get the dissipation rate of $\|\nabla^{m+1}v \|_{L^2}^2$, we notice from Lemma \ref{lem_v} and the induction hypothesis that
\begin{align}\label{est_vm}
\begin{aligned}
&\frac{d}{dt}\|\nabla^m v\|_{L^2}^2 + \frac{1}{3}\|\nabla^m v\|_{L^2}^2 + C_1\mu\|\nabla^{m+1}v\|_{L^2}^2\cr
&\quad \leq C\epsilon_1\|\nabla^m (h,n,v)\|_{L^2}^2 +3\|\nabla^m u\|_{L^2}^2 +C\|\nabla n\|_{H^{m-1}}^2\cr
&\quad \leq C\mh_0(m).
\end{aligned}
\end{align}
Then we multiply \eqref{est_vm} by $\alpha_1 > 0$ and add it to \eqref{est_hum} to find
$$\begin{aligned}
&\frac{d}{dt}\lt(\int_\om |\nabla^{m+1}(h,u)|^2\,dx + \int_\om \nabla^m u \cdot \nabla^{m+1} h \,dx + \int_\om |\nabla^m u|^2\,dx + \alpha_1\int_\om |\nabla^m v|^2\,dx\rt)\cr
&\quad \leq -\lt(\frac14 - C\epsilon_1\rt)\|\nabla^{m+1}(h,u)\|_{L^2}^2 - (C_1\alpha_1\mu - 2)\|\nabla^{m+1}v\|_{L^2}^2 - C\alpha_1\|\nabla^m v\|_{L^2}^2 + C\mh_0(m)\cr
&\quad \leq - C_3\lt( \|\nabla^{m+1}(h,u)\|_{L^2}^2 + \|\nabla^m (u,v)\|_{L^2}^2\rt) + C\mh_0(m),
\end{aligned}$$
for some positive constant $C_3$, where $\alpha_1 > 0$ is chosen large enough so that $C_1\alpha_1 \mu - 2 > 0$. Then by applying the Gronwall's inequality together with Remark \ref{rmk_eqv} we find
\bq\label{res_hu}
\sup_{0 \leq t \leq T}\|\nabla^{m+1}(h,u)\|_{L^2}^2 \leq C\mh_0(m+1),
\eq
for some positive constant $C > 0$. 

We next estimate $\|\nabla^{m+1}(n,v)\|_{L^2}^2$. It follows from Lemma \ref{lem_n} together with \eqref{est_vm} that
\begin{align}\label{est_nm1}
\begin{aligned}
&\frac{d}{dt}\lt(\int_{\om} \nabla^{m+1} n \cdot \lt( \frac{\nabla^{m+1} n}{2} + \frac{(n+1)^2}{2\mu + \lambda} \nabla^{m} v\rt) dx+\alpha_2\int_\om |\nabla^m v|^2\,dx\rt)\cr
&\quad \leq -(C_1\alpha_2\mu - C)\|\nabla^{m+1}v\|_{L^2}^2 - C_2\|\nabla^{m+1}n\|_{L^2}^2 + C\epsilon_1 \|\nabla^{m+2}v\|_{L^2}^2 + C\mh_0(m)\cr
&\quad \leq - C_4\|\nabla^{m+1}(n,v)\|_{L^2}^2 + C\epsilon_1\|\nabla^{m+2}v\|_{L^2}^2 + C\mh_0(m),
\end{aligned}
\end{align}
for some positive constant $C_4$, where $\alpha_2 > 0$ is taken so that $C_1\alpha_2 \mu - C > 0$. On the other hand, we also find from Lemma \ref{lem_v} that
\begin{align}\label{est_vm1}
\begin{aligned}
&\frac{d}{dt}\|\nabla^{m+1} v\|_{L^2}^2 + \frac{1}{3}\|\nabla^{m+1} v\|_{L^2}^2 + C_1\mu\|\nabla^{m+2}v\|_{L^2}^2\cr
&\quad \leq C\epsilon_1\|\nabla^{m+1} (h,n,v)\|_{L^2}^2 +3\|\nabla^{m+1} u\|_{L^2}^2 +C\|\nabla n\|_{H^{m}}^2\cr
&\quad \leq C\epsilon_1\|\nabla^{m+1} v\|_{L^2}^2 + C\|\nabla^{m+1} n\|_{L^2}^2 + C\mh_0(m+1).
\end{aligned}
\end{align}
Multiplying both side of inequality \eqref{est_vm1} by $\alpha_3 > 0$ which will be determined later and combining it with \eqref{est_nm1}, we deduce
$$\begin{aligned}
&\frac{d}{dt}\lt(\int_{\om} \nabla^{m+1} n \cdot \lt( \frac{\nabla^{m+1} n}{2} + \frac{(n+1)^2}{2\mu + \lambda} \nabla^{m} v\rt) dx+\alpha_2\int_\om |\nabla^m v|^2\,dx + \alpha_3 \int_\om |\nabla^{m+1}v|^2\,dx\rt)\cr
&\quad \leq - (C_4 - C\alpha_3)\|\nabla^{m+1}n\|_{L^2}^2 -(C_1\alpha_3\mu -C\epsilon_1)\|\nabla^{m+2}v\|_{L^2}^2 \cr
&\qquad  - \alpha_3\lt(\frac{1}{3} - C\epsilon_1\rt)\|\nabla^{m+1} v\|_{L^2}^2+ C\mh_0(m+1).
\end{aligned}$$
For notational simplicity, we set
\[
E_{m+1} := \int_{\om} \nabla^{m+1} n \cdot \lt( \frac{\nabla^{m+1} n}{2} + \frac{(n+1)^2}{2\mu + \lambda} \nabla^{m} v\rt) dx+\alpha_2\int_\om |\nabla^m v|^2\,dx + \alpha_3 \int_\om |\nabla^{m+1}v|^2\,dx,
\]
and choose $\alpha_3 > 0$ small enough such that $C_4 - C\alpha_3 > 0$ and then select $\epsilon_1 > 0$ such that $C_1\alpha_3 \mu - C\epsilon_1 > 0$ and $1/3 - C\epsilon_1 > 0$. This yields 
\[
\frac{d}{dt}E_{m+1} \leq -C_5\|\nabla^{m+1}(n,v)\|_{L^2}^2 - C_6\|\nabla^{m+2} v\|_{L^2}^2 - C_7\|\nabla^m v\|_{L^2}^2 + C\mh_0(m+1),
\]
for some positive constants $C_5, C_6$, and $C_7$. Note that there exists a positive constant $C > 0$ such that
\bq\label{est_eqv_nv}
\frac1C \lt(\|\nabla^{m+1}(n,v)\|_{L^2}^2 + \|\nabla^m v\|_{L^2}^2\rt) \leq E_{m+1} \leq C \lt(\|\nabla^{m+1}(n,v)\|_{L^2}^2 + \|\nabla^m v\|_{L^2}^2\rt),
\eq
for $\alpha_2 > 0$ large enough due to 
\[
\lt|\int_\om \nabla^{m+1}n \cdot \frac{(n+1)^2}{2\mu + \lambda}\nabla^m v\,dx\rt| \leq \frac14\|\nabla^{m+1} n\|_{L^2}^2 + C\|\nabla^m v\|_{L^2}^2.
\]
Thus by applying the Gronwall's inequality together with the relation \eqref{est_eqv_nv} we have
\bq\label{res_nv}
\sup_{0 \leq t \leq T}\|\nabla^{m+1}(n,v)\|_{L^2}^2 \leq C\mh_0(m+1).
\eq
Hence we conclude our desired result by combining \eqref{res_hu} and \eqref{res_nv}.
\end{proof}
We now construct the global classical solutions by combining the local existence theory in Theorem \ref{thm:local-ext} and the uniform {\it a priori} estimate in Proposition \ref{prop_gd}.
\begin{proof}[Proof of Theorem \ref{thm_ext}] We first choose a positive constant $M := \min\{ \epsilon_0, \epsilon_1\}$, where $\epsilon_0$ and $\epsilon_1$ are given in Theorem \ref{thm:local-ext} and Proposition \ref{prop_gd}, respectively. We then take the initial data $(h_0,u_0,n_0,v_0)$ satisfying
\[
\|(h_0,u_0,n_0,v_0)\|_{H^s} \leq \frac{M}{2\sqrt{1 + C_0}},
\]
where $C_0$ is appeared in Proposition \ref{prop_gd}. We now define the lifespan of the solutions for the system \eqref{main_eq}-\eqref{ini_main_eq} as
\bq\label{def_t}
T := \sup\lt\{ t \geq 0 \,:\, \sup_{0 \leq \tau \leq t} \|(h(\tau),u(\tau),n(\tau),v(\tau))\|_{H^s} < M  \rt\}
\eq
Note that 
\[
\|(h_0,u_0,n_0,v_0)\|_{H^s} \leq \frac{M}{2\sqrt{1 + C_0}} \leq \frac{M}{2} \leq \epsilon_0.
\]
This and together with Theorem \ref{thm:local-ext} implies $T > 0$. Suppose $T < +\infty$. Then it follows from the definition of $T$ \eqref{def_t} and Theorem \ref{thm:local-ext} that
\bq\label{contra}
\sup_{0 \leq t \leq T} \|(h(t),u(t),n(t),v(t))\|_{H^s} = M.
\eq
On the other hand, we find from Proposition \ref{prop_gd} that
\[
\sup_{0 \leq t \leq T} \|(h(t),u(t),n(t),v(t))\|_{H^s} \leq \sqrt{C_0}\|(h_0,u_0,n_0,v_0)\|_{H^s} \leq \frac{M\sqrt{C_0}}{2\sqrt{1 + C_0}} \leq \frac{M}{2},
\]
and this is a contradiction to \eqref{contra}. Hence we conclude that the lifespan of the solutions $T = \infty$ and complete the proof.

\end{proof}

%
%
%
%
\section{Large-time behavior}\label{sec:lg}
In this section, we study the large-time behavior of classical solutions to the system \eqref{f_eq}. As mentioned in Section \ref{sec:intro}, inspired by \cite{Choi1,Choi2}, we first define a temporal energy function $\me$ as
\begin{align}\label{tem_e}
\begin{aligned}
\me(t) &:= \frac12\intt \rho|u - m_c|^2\,dx + \intt \rho \int_{\rho_c}^\rho \frac{s-\rho_c}{s^2}\,dsdx + \frac12\intt n|v - j_c|^2\,dx \cr
&\quad + \intt n \int^n_{n_c} \frac{s^\gamma - n_c^\gamma}{s^2}\,ds dx + \frac12\lt( \frac{n_c\,\rho_c}{n_c + \rho_c}\rt)|m_c - j_c|^2.
\end{aligned}
\end{align}
In the lemma below, we show that $\me$ has the same dissipation with the total energy $E$ defined in Lemma \ref{lem_energy}.
\begin{lemma}Let $(\rho,u,n,v)$ be any classical solutions to the system \eqref{f_eq}. Then we have
\[
\frac{d}{dt}\me(t) + \md(t) = 0,
\]
where the corresponding dissipation $\md(t)$ is given by
\[
\md(t) := \mu\intt |\nabla v|^2\,dx + (\mu + \lambda)\intt |\nabla \cdot v|^2\,dx + \intt \rho|u - v|^2\,dx.
\]
\end{lemma}
\begin{proof}It is a straightforward to get
$$\begin{aligned}
&\frac{d}{dt}\lt(\frac12 \intt \rho|u-m_c|^2\,dx + \intt \rho \ln \rho\,dx\rt) = -\intt \rho(u-v) \cdot (u - m_c)\,dx\cr
&\frac{d}{dt}\lt(\frac12 \intt n| v - j_c|^2\,dx + \frac{1}{\gamma - 1}\intt n^\gamma\,dx\rt) + \mu\intt |\nabla v|^2\,dx + (\mu + \lambda)\intt |\nabla \cdot v|^2\,dx\cr
&\qquad = \intt \rho(u-v)\cdot (v - j_c)\,dx.
\end{aligned}$$
This and together with Lemma \ref{lem_press} yields
\begin{align}\label{est_en1}
\begin{aligned}
&\frac{d}{dt}\lt(\frac12 \intt \rho|u-m_c|^2\,dx + \intt \rho \int_{\rho_c}^\rho \frac{s-\rho_c}{s^2}\,dsdx+ \frac12 \intt n| v - j_c|^2\,dx + \intt n \int^n_{n_c} \frac{s^\gamma - n_c^\gamma}{s^2}\,ds dx\rt)\cr
&\qquad + \mu\intt |\nabla v|^2\,dx + (\mu + \lambda)\intt |\nabla \cdot v|^2\,dx + \intt \rho|u-v|^2\,dx \cr
&\qquad \qquad = (m_c - j_c)\cdot\intt \rho(u-v)\,dx.
\end{aligned}
\end{align}
On the other hand, it follows from the conservation of total momentum that
\bq\label{est_tm}
\frac{d}{dt}\lt( \rho_c m_c\rt) + \frac{d}{dt}\lt( n_c j_c\rt) = 0, \quad \mbox{i.e.,} \quad \frac{d}{dt} m_c = -\frac{n_c}{\rho_c}\frac{d}{dt}j_c.
\eq
Then we obtain
\begin{align}\label{est_en2}
\begin{aligned}
\frac12\frac{d}{dt}|m_c - j_c|^2 &= (m_c - j_c)\cdot \frac{d}{dt}(m_c - j_c) \cr
&= - \lt(1 + \frac{n_c}{\rho_c} \rt)(m_c - j_c) \cdot \frac{d}{dt} j_c \cr
&= -\lt(\frac{1}{n_c} + \frac{1}{\rho_c} \rt)(m_c - j_c) \cdot \intt \rho(u-v)\,dx
\end{aligned}
\end{align}
Combining \eqref{est_en1} and \eqref{est_en2}, we conclude our desired result.
\end{proof}
\begin{remark}\label{rmk_eqvv}1. We set 
\[
\widetilde E(t) := \int_\om \rho|u|^2 dx + \int_\om (\rho - \rho_c)^2 dx + \int_\om n|v|^2 dx + \int_\om (n - n_c)^2 dx.
\]
Then by using the almost same argument as in Remark \ref{rmk:energy} we get
\[
\widetilde E(t) \leq C\widetilde E_0 \quad \mbox{for} \quad t \geq 0,
\]
where $\widetilde E_0 = \widetilde E(0)$ and $C$ is a positive constant independent of $t$.

2. We reminder the reader that our Lyapunov function $\ml$ is given by
\[
\ml(t)= \int_{\T^3} \rho |u-m_c|^2 dx + \int_{\T^3} (\rho - \rho_c)^2 dx + \int_{\T^3} n| v - j_c|^2 dx + |m_c - j_c|^2 + \int_{\T^3} (n-n_c)^2 dx.
\]
Since it follows from Lemma \ref{lem_press} that 
\[
\intt \rho \int_{\rho_c}^\rho \frac{s-\rho_c}{s^2}\,dsdx+ \intt n \int^n_{n_c} \frac{s^\gamma - n_c^\gamma}{s^2}\,ds dx \approx \int_{\T^3} (\rho - \rho_c)^2\,dx +\int_{\T^3} (n-n_c)^2 dx,
\]
$\ml(t) \approx \me(t)$ for all $t \geq 0$, i.e., there exists a $C > 0$ independent of $t$ such that 
\[
\frac{1}{C} \me(t) \leq \ml(t) \leq C\me(t) \quad \mbox{for} \quad t \geq 0.
\]
\end{remark}
We next present several estimates for averaged quantities which will be frequently used in the rest of this section. 
\begin{lemma}\label{lem_mom}Let $m_c$ and $j_c$ be the local momenta of the fluids defined in \eqref{def1}. Then we have
$$\begin{aligned}
&(i)\,\,\,\, |m_c|^2 + |j_c|^2 \leq C\widetilde E_0 \quad \mbox{and} \quad |m_c^\prime|^2 + |j_c^\prime|^2 \leq C\intt \rho|u-v|^2dx,\cr
&(ii) \,\,|v_c - j_c|^2 \leq C\intt |\nabla v|^2\,dx, \quad \mbox{where} \quad v_c := \intt v\,dx,
\end{aligned}$$
where $^\prime$ denotes the time-derivative, i.e., $\{\}^\prime := \frac{d}{dt}\{\}$, and $C$ is a positive constant independent of $t$.
\end{lemma}
\begin{proof}Using the H\"older's inequality, we easily get
\[
|m_c| \leq \frac{1}{\rho_c}\intt \rho|u|\,dx \leq \frac{1}{\sqrt{\rho_c}}\lt(\intt \rho|u|^2\,dx\rt)^{1/2} \leq C\wt E_0^{1/2},
\]
due to Remark \ref{rmk_eqvv}, and similarly we obtain $|j_c|^2 \leq C\wt E_0$.

We also notice from \eqref{est_tm} that
\[
|\rho_c \,m_c^\prime|^2 = |n_c \,j_c^\prime|^2 = \lt| \intt \rho (u - v)\,dx\rt|^2 \leq \rho_c \intt \rho|u-v|^2dx,
\]
and this implies
\[
|m_c^\prime|^2 + |j_c^\prime|^2 \leq \lt(\frac{1}{\rho_c} + \frac{\rho_c}{n_c^2} \rt)\intt \rho |u-v|^2 dx.
\]
For the proof of $(ii)$, we use the Sobolev inequality to find
\[
|v_c - j_c|^2 = \frac{1}{n_c^2}\lt|\intt n(v - v_c)\,dx\rt|^2 \leq \frac{\bar n}{n_c}\intt |v - v_c|^2 dx \leq \frac{C\bar n}{n_c}\intt |\nabla v|^2 dx.
\]
This concludes the desired results.
\end{proof}

We provide the lower bound estimate of the dissipation term $\md$. In fact, our proposed Lyapunov functional $\ml$ without the evolution functions of the densities is bounded from above by the dissipation term $\md$. From this observation, we will consider a perturbed energy functional $\me^\sigg$ to have the dissipation rates for the densities  $\rho$ and $n$ later (see \eqref{f_per}).

\begin{lemma}\label{lem_gg}There exists a positive constant $C$ independent of $t$ such that 
\[
\ml_-(t) \leq C\md(t) \quad \mbox{for} \quad t\geq 0,
\]
where $\ml_-$ is given by
\[
\ml_- := \ml - \intt (\rho - \rho_c)^2\,dx - \intt (n-n_c)^2\,dx.
\]
\end{lemma}
\begin{proof}Using the standard interpolation technique and Young's inequality, we obtain
\begin{align*}
\begin{aligned}
\int_{\T^3} \rho|u-v|^2 dx &= \int_{\T^3} \rho|u-m_c + m_c - j_c + j_c - v|^2 dx\cr
&= \int_{\T^3} \rho|u-m_c|^2 dx + \rho_c|m_c - j_c|^2 + \int_{\T^3} \rho|v-j_c|^2 dx\cr
&\quad + 2\int_{\T^3} \rho(m_c - j_c)\cdot (j_c - v) dx + 2\int_{\T^3} \rho (u - m_c) \cdot (j_c - v) dx\cr
&\geq \frac12\int_{\T^3} \rho|u-m_c|^2 dx + \frac{\rho_c}{2}|m_c - j_c|^2 - 3\int_{\T^3} \rho |v - j_c|^2 dx.
\end{aligned}
\end{align*}
On the other hand, the negative term in right hand side of the above inequality can be estimated as
\[
\intt \rho | v - j_c|^2\,dx \leq 2\bar \rho\intt |v - v_c|^2\,dx + 2\rho_c|v_c - j_c|^2 \leq C\intt |\nabla v|^2 \,dx,
\]
due to Lemma \ref{lem_mom}. Then this yields
\begin{align}\label{lem:equiv2}
\begin{aligned}
\frac12\lt( \int_{\T^3} \rho|u - m_c|^2 dx + \rho_c |m_c - j_c|^2\rt) 
&\leq \int_{\T^3} \rho |u-v|^2 dx + C\int_{\T^3} |\nabla v|^2 dx.
\end{aligned}
\end{align}
We next estimate 
\begin{align}\label{lem:equiv1}
\begin{aligned}
\frac12\int_{\T^3} n|v - j_c|^2 dx &\leq \int_{\T^3} n|v-v_c|^2 dx + n_c|v_c - j_c|^2 \leq C\int_{\T^3} |\nabla v|^2 dx,
\end{aligned}
\end{align}
where we used Lemma \ref{lem_mom} again.

Combine the estimates \eqref{lem:equiv1} and \eqref{lem:equiv2}, we finally conclude
\begin{align*}
\begin{aligned}
& \frac12\lt( \int_{\T^3} \rho |u - m_c|^2 dx + \rho_c |m_c - j_c|^2 + \int_{\T^3} (n+1)|v - j_c|^2 dx\rt)\cr
& \qquad \leq \int_{\T^3} \rho |u-v|^2 dx + C\int_{\T^3} |\nabla v|^2 dx,
\end{aligned}
\end{align*}
that is,
\[
\mathcal{L}_-(t) \leq C\mathcal{D}(t), \quad \mbox{for} \quad t \geq 0,
\]
where $C$ is a positive constant independent of $t$.
\end{proof}
We next focus on finding the desired dissipation rates for the densities $\rho$ and $n$. Inspired by the recent works \cite{Choi1, Choi2, CK}, we provide a type of Bogovskii's result in the periodic domain which is obtained from the estimate of elliptic regularity for Poisson's equations. For more details, we refer to \cite{Choi2, Gal}.
\begin{lemma}\label{lem_bogo}Given any $f \in L^2_{\#}(\T^3) := \Big\{ f \in L^2(\T^3) | \int_{\T^3} f dx = 0
\Big\}$, the following stationary transport equation with auxiliary equations
\bq\label{tran_eq}
\nabla \cdot \nu = f, \quad \nabla \times \nu = 0, \quad \mbox{and} \quad \int_{\T^3} \nu\,dx = 0,
\eq
admit a solution operator $\mb: f \mapsto \nu$ satisfying the following properties:
\begin{enumerate}
\item[(i)]  $\nu = \mb[f]$ is a solution to the problem \eqref{tran_eq} and a linear operator from $L^2_{\#}(\T^3)$ into $H^1(\T^3)$, i.e.,
\[
\|\mathcal{B}[f] \|_{H^1} \leq C^*\| f \|_{L^{2}}.
\]

\item[(ii)]  If a function $f \in H^1(\T^3)$ can be written in the form
$f = \nabla \cdot g$ with $g \in [H^1(\T^3)]^3$,
    then
    \[
    \| \mathcal{B}[f] \|_{L^{2}} \leq C^*\|g\|_{L^{2}}.
    \]
\end{enumerate}
\end{lemma}
Using the operator $\mb$, we set a perturbed function $\mathcal{E}^{\sigma_1,\sigma_2}(t)$:
\bq\label{f_per}
\me^\sigg := \me - \sigma_1 \intt \rho(u - m_c)\cdot \mb [\rho - \rho_c]\,dx - \sigma_2 \intt n(v - j_c)\cdot \mb[n - n_c]\,dx,
\eq
where $\sigma_1, \sigma_2 > 0$ will be appropriately determined later.
\begin{lemma}The perturbed function $\me^\sigg$ defined in \eqref{f_per} satisfies
\[
\frac{d}{dt}\me^\sigg(t) + \md^\sigg(t) = 0,
\]
where $\md^\sigg$ is given by
$$\begin{aligned}
\md^\sigg &= \md + \sigma_1\lt( \intt \rho u \otimes u : \nabla \mb[\rho - \rho_c]\,dx - \intt \rho(u-v) \cdot \mb[\rho - \rho_c]\,dx + \intt (\rho - \rho_c)^2\,dx\rt)\cr
&\quad -\sigma_1\lt( \intt \rho (u - m_c)\cdot \mb[\nabla \cdot (\rho u)]\,dx + \intt \pa_t(\rho m_c) \cdot \mb[\rho - \rho_c]\,dx\rt)\cr
&\quad + \sigma_2 \lt( \intt nv \otimes v : \nabla \mb[n - n_c]\,dx + \intt \rho(u-v) \cdot \mb[n - n_c]\,dx + \intt (n^\gamma - n_c^\gamma)(n - n_c)\,dx\rt)\cr
&\quad -\sigma_2 \lt( \intt n(v - j_c)\cdot \mb[\nabla \cdot(nv)]\,dx + \mu\intt \nabla v: \nabla \mb[n - n_c]\,dx + (\mu + \lambda)\intt (\nabla \cdot v)n\,dx\rt)\cr
&\quad -\sigma_2\intt \pa_t(vj_c)\cdot \mb[n - n_c]\,dx,
\end{aligned}$$
where $A:B := \sum_{i=1}^m \sum_{j=1}^n a_{ij}b_{ij}$ for $A = (a_{ij}), B = (b_{ij}) \in \R^{mn}$.
\end{lemma}
\begin{proof}
We first divide the second term in the right hand side of the equation \eqref{f_per} into three parts:
$$\begin{aligned}
\frac{d}{dt}\intt \rho(u - m_c)\cdot \mb[\rho - \rho_c]\,dx &= \intt \pa_t(\rho u) \cdot \mb[\rho - \rho_c]\,dx + \intt \rho(u - m_c)\cdot \mb[\pa_t \rho]\,dx\cr
&\quad - \intt \pa_t (\rho m_c)\cdot\mb[\rho - \rho_c]\,dx\cr
&=: \sum_{i=1}^3 I_i,
\end{aligned}$$
where $I_1$ is estimated as 
$$\begin{aligned}
I_1 &= -\intt \lt(\nabla \cdot (\rho u \otimes u) + \rho(u-v) + \nabla (\rho - \rho_c) \rt)\cdot \mb[\rho - \rho_c]\,dx\cr
&= \intt \rho u \otimes u : \mb[\rho - \rho_c]\,dx - \intt \rho(u-v)\cdot \mb[\rho - \rho_c]\,dx + \intt (\rho - \rho_c)^2\,dx,
\end{aligned}$$
due to the definition of the operator $\mb$. The estimate of $I_2$ is easily obtained by the continuity equation $\eqref{f_eq}_1$. By using the almost same argument as the above, we also find the desired estimates of the last term in the right hand side of the equation \eqref{f_per}. This concludes the proof.
\end{proof}

Note that for $C^* > 0$ which is appeared in Lemma \ref{lem_bogo} we obtain
\[
\sigma_1\lt|\intt \rho (u - m_c) \cdot \mb[\rho - \rho_c]\,dx \rt| \leq \sigma_1\lt(\frac{\bar\rho}{2}\intt \rho|u - m_c|^2\,dx + \frac{C^*}{2}\intt (\rho - \rho_c)^2\,dx \rt),
\]
and
\[
\sigma_2\lt|\intt n (v - j_c) \cdot \mb[n - n_c]\,dx \rt| \leq \sigma_1\lt(\frac{\bar n}{2}\intt n|v - j_c|^2\,dx + \frac{C^*}{2}\intt (n - n_c)^2\,dx \rt).
\]
This and together with Remark \ref{rmk_eqvv} deduces
\bq\label{g_ob}
\me^\sigg(t) \approx \ml(t) \quad t \geq 0,
\eq
for some $\sigma_1, \sigma_2 > 0$ small enough. 
We now show that the dissipation term $\md^{\sigma_1,\sigma_2}$ is bounded from below by the Lyapunov function $\ml$ from which we deduce the exponential decay of $\ml$.
\begin{proof}[Proof of Theorem \ref{thm_large}] For the proof, we claim that there exists a positive constant $C$ such that
\bq\label{claim}
\ml(t) \leq C\md^\sigg(t) \quad t \geq 0,
\eq
for $\sigma_1, \sigma_2 > 0$ small enough. We notice that if the inequality \eqref{claim} holds, then it follows from \eqref{g_ob} that 
\[
\frac{d}{dt}\me^\sigg(t) + C\me^\sigg(t) \leq 0 \quad \mbox{for some} \quad C>0,
\]
and this concludes
\[
\ml(t) \ls \me^\sigg(t) \ls \me^\sigg_0 e^{-Ct} \ls \ml_0 e^{-Ct} \quad \mbox{for} \quad t \geq 0.
\]
{\bf Proof of claim:} We rewrite $\md^\sigg(t)$ as the summation of $J_i,i=1,\cdots, 15$:
$$\begin{aligned}
\md^\sigg &= \mu\intt |\nabla v|^2\,dx + (\mu + \lambda)\intt |\nabla \cdot v|^2\,dx + \intt \rho|u - v|^2\,dx\cr
&\quad + \sigma_1\lt( \intt \rho u \otimes u : \nabla \mb[\rho - \rho_c]\,dx - \intt \rho(u-v) \cdot \mb[\rho - \rho_c]\,dx + \intt (\rho - \rho_c)^2\,dx\rt)\cr
&\quad -\sigma_1\lt( \intt \rho (u - m_c)\cdot \mb[\nabla \cdot (\rho u)]\,dx + \intt \pa_t(\rho m_c) \cdot \mb[\rho - \rho_c]\,dx\rt)\cr
&\quad + \sigma_2 \lt( \intt nv \otimes v : \nabla \mb[n - n_c]\,dx + \intt \rho(u-v) \cdot \mb[n - n_c]\,dx + \intt (n^\gamma - n_c^\gamma)(n - n_c)\,dx\rt)\cr
&\quad -\sigma_2 \lt( \intt n(v - j_c)\cdot \mb[\nabla \cdot(nv)]\,dx +\intt \pa_t(vj_c)\cdot \mb[n - n_c]\,dx+ \mu\intt \nabla v: \nabla \mb[n - n_c]\,dx \rt)\cr
&\quad -\sigma_2(\mu + \lambda)\intt (\nabla \cdot v)(n - n_c)\,dx\cr
&=:\sum_{i=1}^{15}J_i.
\end{aligned}$$
We next estimate each term as follows separately. \newline

\noindent $\diamond$ Estimate of $J_4$: We decompose $J_4$ into three parts as
$$\begin{aligned}
J_4 &= \sigma_1\intt \rho (u - m_c)\otimes u : \nabla \mb[\rho - \rho_c]\,dx + \sigma_1\intt \rho m_c \otimes (u - m_c) : \nabla \mb[\rho - \rho_c]\,dx \cr
&\quad + \sigma_1 \intt (\rho - \rho_c)m_c\otimes m_c : \nabla \mb[\rho - \rho_c]\,dx\cr
&=: \sum_{i=1}^3 J_4^i.
\end{aligned}$$
Then we make use of Lemma \ref{lem_bogo} to get
$$\begin{aligned}
J_4^1 &\leq \frac{\sigma_1^{1/2}\bar\rho \|u\|_{L^\infty}}{2}\intt \rho|u-m_c|^2\,dx + C^* \sigma_1^{3/2}\intt (\rho - \rho_c)^2\,dx,\cr
J_4^2 &\leq C\sigma_1^{1/2}\bar\rho \wt E_0\intt \rho |u - m_c|^2\,dx + C^* \sigma_1^{3/2}\intt (\rho - \rho_c)^2\,dx,\cr
J_4^3 &\leq C^*\sigma_1 \wt E_0 \intt (\rho - \rho_c)^2\,dx,
\end{aligned}$$
where $\wt E_0$ is defined in Remark \ref{rmk_eqvv}.
Thus we obtain
\[
J_4 \leq C\sigma_1^{1/2}\intt \rho |u - m_c|^2\,dx + C\sigma_1\lt( \wt E_0 + \sigma_1^{1/2}\rt)\intt (\rho - \rho_c)^2\,dx,
\]
where $C$ is a positive constant independent of $\sigma_1$ and $t$. \newline

\noindent $\diamond$ Estimate of $J_5$: It is a straightforward to get
\[
J_5 \leq \frac14\intt \rho |u -v|^2\,dx + C^*\sigma_1^2 \bar \rho \intt (\rho - \rho_c)^2\,dx.
\]
$\diamond$ Estimate of $J_7$: By adding and subtracting together with Lemma \ref{lem_bogo}, we get
$$\begin{aligned}
J_7 &= -\sigma_1\intt \rho(u - m_c)\cdot \mb[\nabla \cdot (\rho(u - m_c))]\,dx - \sigma_1\intt \rho(u - m_c)\cdot \mb[\nabla \cdot ((\rho - \rho_c)m_c)]\,dx\cr
&\leq C^* \sigma_1 \bar \rho \intt \rho|u - m_c|^2\,dx + \frac{\sigma_1^{1/2}\bar\rho}{2}\intt \rho |u - m_c|^2\,dx + \frac{C^*\sigma_1^{3/2}|m_c|^2}{2}\intt (\rho - \rho_c)^2\,dx\cr
&\leq C\sigma_1^{1/2}\intt \rho|u - m_c|^2\,dx + C\sigma_1^{3/2}\intt (\rho - \rho_c)^2\,dx,
\end{aligned}$$
where $C$ is a positive constant independent of $\sigma_1$ and $t$. \newline

\noindent $\diamond$ Estimate of $J_8$: It follows from the continuity equation $\eqref{main_eq}_1$ that
\[
\pa_t (\rho m_c) = -\nabla \cdot (\rho u) m_c + \rho m_c^\prime.
\]
Then by using the integration by parts, we divide $J_8$ into two parts as
$$\begin{aligned}
J_8 &= -\sigma_1\intt m_c \cdot (\rho u \cdot \nabla)\mb[\rho - \rho_c]\,dx - \sigma_1\intt \rho m_c^\prime \cdot \mb[\rho - \rho_c]\,dx\cr
&=: J_8^1 + J_8^2.
\end{aligned}$$
Here $J_8^i,i=1,2$ are estimated as follows.
$$\begin{aligned}
J_8^1 &= -\sigma_1 \intt m_c \cdot (\rho (u - m_c)\cdot \nabla) \mb[\rho - \rho_c]\,dx - \sigma_1 \intt m_c \cdot ((\rho - \rho_c)m_c \cdot \nabla)\mb[\rho - \rho_c]\,dx\cr
&\leq C\sigma_1^{1/2}\bar\rho \wt E_0 \intt \rho|u-m_c|^2\,dx + C^* \sigma_1^{3/2}\intt (\rho - \rho_c)^2\,dx + C C^*\sigma_1\wt E_0 \intt (\rho - \rho_c)^2\,dx\cr
&= C\sigma_1^{1/2}\bar\rho \wt E_0\intt \rho|u - m_c|^2\,dx + C^*\sigma_1\lt( C\wt E_0 + \sigma_1^{1/2}\rt)\intt (\rho - \rho_c)^2\,dx,\cr
J_8^2 &\leq \frac{\sigma_1^{1/2}\bar\rho^2}{2}|m_c^\prime|^2 + C^*\sigma_1^{3/2}\intt (\rho - \rho_c)^2\,dx,
\end{aligned}$$
due to Lemma \ref{lem_bogo}. This and together with Lemma \ref{lem_mom} yields
\[
J_8 \leq C\sigma_1^{1/2}\intt \rho |u - m_c|^2\,dx + C\sigma_1\lt(\wt E_0 + \sigma_1^{1/2} \rt)\intt (\rho - \rho_c)^2\,dx + C\sigma_1^{1/2}\intt \rho |u - v|^2\,dx,
\]
where $C$ is a positive constant independent of $\sigma_1$ and $t$. \newline

\noindent $\diamond$ Estimate of $J_i, i=9,10,12,13$: Using similar strategies as the above, we deduce
$$\begin{aligned}
J_9 &\leq C\sigma_2^{1/2}\intt n |v - j_c|^2\,dx + C\sigma_2\lt( \wt E_0 + \sigma_2^{1/2}\rt)\intt (n - n_c)^2\,dx,\cr
J_{10} &\leq \frac14\intt \rho |u - v|^2\,dx + C^* \sigma_2^2 \bar \rho \intt (n - n_c)^2\,dx,\cr
J_{12} &\leq C\sigma_2^{1/2}\intt n|v - j_c|^2\,dx + C\sigma_2^{3/2}\intt (n - n_c)^2\,dx,\cr
J_{13} &\leq C\sigma_2^{1/2}\intt n |v - j_c|^2\,dx + C\sigma_2\lt(\wt E_0 + \sigma_2^{1/2} \rt)\intt (n - n_c)^2\,dx + C\sigma_2^{1/2}\intt \rho |u - v|^2\,dx,
\end{aligned}$$
where $C$ is a positive constant independent of $\sigma_2$ and $t$. \newline

\noindent $\diamond$ Estimate of $J_i,i=11,14,15$: We easily estimate as
$$\begin{aligned}
J_{11} &\geq C\sigma_2\intt (n - n_c)^2\,dx,\cr
J_{14} &\leq \frac{\mu}{2}\intt |\nabla v|^2\,dx + C^*\sigma_2^2 \mu \intt (n - n_c)^2\,dx,\cr
J_{15} &\leq \frac{\mu + \lambda}{2}\intt |\nabla \cdot v|^2\,dx + C^*\sigma_2^2(\mu + \lambda)\intt (n - n_c)^2\,dx,
\end{aligned}$$
where $C$ is a positive constant independent of $\sigma_2$ and $t$. \newline

We now combine all of the above estimates to find
$$\begin{aligned}
\md^\sigg(t) &\geq \frac{\mu}{2}\intt |\nabla v|^2\,dx + \lt(\frac12 - C\lt(\sigma_1^{1/2} + \sigma_2^{1/2}\rt) \rt)\intt \rho | u-v|^2\,dx\cr
&\quad + \sigma_1\lt(1 - C\lt( \wt E_0 + \sigma_1^{1/2} +\sigma_1\rt) \rt)\intt (\rho - \rho_c)^2\,dx \cr
&\quad + C\sigma_2\lt(1 - C\lt( \wt E_0 + \sigma_2^{1/2} +\sigma_2\rt) \rt)\intt (n - n_c)^2\,dx.
\end{aligned}$$
For the sake of simplicity, we set 
\[
C_1:= \frac12 - C\lt(\sigma_1^{1/2} + \sigma_2^{1/2}\rt), \quad C_2:= \sigma_1\lt(1 - C\lt( \wt E_0 + \sigma_1^{1/2} +\sigma_1\rt) \rt),
\]
and
\[
C_3:= C\sigma_2\lt(1 - C\lt( \wt E_0 + \sigma_2^{1/2} +\sigma_2\rt) \rt).
\]
We now choose $\sigma_1, \sigma_2 > 0$, and $\wt E_0$ small enough so that the constants $C_i > 0$ for all $i= 1,2,3$. Finally we use the inequality obtained in Lemma \ref{lem_gg} to have
$$\begin{aligned}
&\md^\sigg(t) \cr
&\quad \geq \frac{\mu}{2}\intt |\nabla v|^2\,dx + C_1\intt \rho|u - v|^2\,dx + C_2\intt (\rho - \rho_c)^2\,dx + C_3\intt (n - n_c)^2\,dx\cr
&\quad \geq \min\lt\{\frac\mu2, C_1\rt\}\lt(\intt |\nabla v|^2\,dx + \intt \rho|u - v|^2\,dx\rt) + C_2\intt (\rho - \rho_c)^2\,dx + C_3\intt (n - n_c)^2\,dx\cr
&\quad \geq \min\lt\{\frac\mu2, C_1\rt\}\ml_-(t)+ C_2\intt (\rho - \rho_c)^2\,dx + C_3\intt (n - n_c)^2\,dx\cr
&\quad \geq \min\lt\{\min\lt\{\frac\mu2, C_1\rt\}, C_2, C_3 \rt\}\ml(t).
\end{aligned}$$
This completes the proof of claim.
\end{proof}

%
%
%
%

\appendix
\section{Proof of Lemma \ref{lem_n} }\label{app_a}
In this part, we provide the details of the proof of Lemma \ref{lem_n}. We reminder the reader that the similar strategy is used in \cite{CK}, but the estimates in \cite{CK} can not be directly applied to the whole space case.

For $1 \leq k \leq s$, it follows from \eqref{main_eq} that
\begin{align*}
\begin{aligned}
&\frac{d}{dt}\int_{\om} \nabla^k n \cdot \lt( \frac{\nabla^k n}{2} + \frac{(n+1)^2}{2\mu + \lambda} \nabla^{k-1} v\rt) dx\cr
&\quad = \int_{\om} \nabla^k n_t \cdot \nabla^k n \,dx + \int_{\om} \frac{(1+n)^2}{2\mu + \lambda}\nabla^k n_t \cdot \nabla^{k-1}v\,dx\cr
&\qquad  + \int_{\om} \frac{(1+n)^2}{2\mu + \lambda}\nabla^k n \cdot \nabla^{k-1} v_t \,dx + \int_{\om} \nabla^{k-1}n \cdot \lt( \frac{2(1+n)}{2\mu + \lambda}n_t\nabla^{k-1}v \rt)dx\cr
&\quad =: \sum_{i=1}^4 I_i.
\end{aligned}
\end{align*}
$\diamond$ Estimate of $I_1$: A straightforward computation yields that
\begin{align}\label{est_n0}
\begin{aligned}
I_1 &= -\int_{\om} \nabla^k n \cdot \nabla^k\lt( \nabla n \cdot v + (1+n)\nabla \cdot v\rt)dx\cr
&= - \int_{\om} \nabla^k n \cdot \lt(v \cdot \nabla^{k+1}n\rt)dx - \int_{\om} [\nabla^k,v\cdot \nabla ]n \cdot \nabla^k n \,dx - \int_{\om} \nabla^k n \cdot \nabla^k \lt( (1+n)\nabla \cdot v\rt)dx\cr
&=:\sum_{i=1}^3 I_1^i,
\end{aligned}
\end{align}
where $I_1^i,i=1,2$ are easily estimated as follows.
\begin{align}\label{est_n1}
\begin{aligned}
I_1^1 &\leq \|\nabla \cdot v\|_{L^\infty}\|\nabla^k n\|_{L^2}^2\leq C\epsilon_1\|\nabla^k n\|_{L^2}^2,\cr
I_1^2 &\leq \|[\nabla^k,v \cdot \nabla]n\|_{L^2}\|\nabla^k n\|_{L^2} \leq C\epsilon_1\|\nabla^k n\|_{L^2}^2,
\end{aligned}
\end{align}
due to Lemma \ref{lem:bern}. In the following estimate of $I_1^3$, for notational simplicity, we omit the summation, i.e., $f_ig_i := \sum_{i=1}^3 f_i g_i$. Using the integration by parts, we estimate $I_1^3$ as
\begin{align}\label{est_n2}
\begin{aligned}
I_1^3 &= -\int_{\om}\nabla^{k-1}\pa_i n \lt( \pa_i n \nabla^{k-1}\pa_j v_j + (n+1)\nabla^{k-1}\pa_{ij}v_j\rt)dx\cr
&\quad - \int_{\om} \nabla^{k-1}\pa_i n \lt( (\nabla^{k-1}\pa_i n)\pa_j v_j + \nabla^{k-1}n \pa_{ij}v_j\rt)dx\cr
&\quad - (1- \delta_{k,1})\sum_{1 \leq l \leq k-2}\binom{k-1}{l}\int_{\om} \nabla^{k-1}\pa_i n \lt( \nabla^l \pa_i n \nabla^{k-1-l}\pa_j v_j + \nabla^l n \nabla^{k-1-l}\pa_{ij}v_j\rt)dx\cr
&\leq C\|\nabla^k n\|_{L^2}\|\nabla n\|_{L^\infty}\|\nabla^k v\|_{L^2} +\int_{\T^3}(n+1)\nabla^{k-1}\pa_{ij} n \cdot \nabla^{k-1}\pa_i v_j\,dx\cr
&\quad  + C\|\nabla^k n\|_{L^2}^2\|\nabla v\|_{L^\infty}+ C\|\nabla^{k-1}n\|_{H^1}^2\|\nabla^2 v\|_{H^1}\cr
&\quad + C(1 - \delta_{k,1})\|\nabla^k n\|_{L^2}\lt(\|\nabla^2 n\|_{H^{k-2}}\|\nabla^2 v\|_{H^{k-2}} + \|\nabla n\|_{H^{k-2}}\|\nabla^3 v\|_{H^{k-2}}\rt)\cr
&\leq C\epsilon_1\|\nabla^k (n,v)\|_{L^2}^2 + C\epsilon_1\|\nabla^{k-1} n\|_{L^2}^2 + C\epsilon_1\|\nabla^2 v\|_{H^{k-1}}^2\cr
&\quad+ \int_{\om}(n+1)\nabla^{k-1}\pa_{ij} n \cdot \nabla^{k-1}\pa_i v_j\,dx,
\end{aligned}
\end{align}
where $\delta_{i,j}$ is the Kronecker dirac function and we used the following estimate for the first term in $I_1^3$.
\begin{align*}
\begin{aligned}
&-\int_{\om}\nabla^{k-1}\pa_i n \lt( \pa_i n \nabla^{k-1}\pa_j v_j + (n+1)\nabla^{k-1}\pa_{ij}v_j\rt)dx\cr
&\quad = -\int_{\om} (\nabla^{k-1}\pa_i n)\pa_i n \nabla^{k-1}\pa_j v_j \,dx + \int_{\om} \lt( (\nabla^{k-1}\pa_{ij}n)(n+1) + (\nabla^{k-1}\pa_i n)\pa_j n \rt)\pa_i v_j\,dx\cr
&\quad \leq C\|\nabla^k n\|_{L^2}\|\nabla n\|_{L^\infty}\|\nabla^k v\|_{L^2} + \int_{\om}(n+1)\nabla^{k-1}\pa_{ij} n \cdot \nabla^{k-1}\pa_i v_j\,dx.
\end{aligned}
\end{align*}
Then by plugging the estimates \eqref{est_n1} and \eqref{est_n2} into \eqref{est_n0}, we obtain
\[
I_1 \leq C\epsilon_1\|\nabla^k (n,v)\|_{L^2}^2 + C\epsilon_1\|\nabla^{k-1} n\|_{L^2}^2 + C\epsilon_1\|\nabla^2 v\|_{H^{k-1}}^2+ \int_{\om}(n+1)\nabla^{k-1}\pa_{ij} n \cdot \nabla^{k-1}\pa_i v_j\,dx.
\]
$\diamond$ Estimate of $I_2$: Using the integration by parts, we get
\begin{align*}
\begin{aligned}
I_2 &= -\frac{1}{2\mu + \lambda}\int_{\om} (1+n)^2 \nabla^{k-1}v \cdot \lt( \nabla^k(\nabla n \cdot v) + \nabla^k((1+n)\nabla\cdot v)\rt) dx\cr
&=\frac{1}{2\mu + \lambda}\int_{\om} \lt( 2(1+n)\nabla n \cdot \nabla^{k-1}v + (1+n)^2\nabla^k v\rt)\cdot \nabla^{k-1}(\nabla n \cdot v)\,dx\cr
&\quad +\frac{1}{2\mu + \lambda}\int_{\om} \lt( 2(1+n)\nabla n \cdot \nabla^{k-1}v + (1+n)^2\nabla^k v\rt)\cdot \nabla^{k-1}((1+n)\nabla\cdot v)\, dx\cr
&=:I_2^1 + I_2^2.\cr
\end{aligned}
\end{align*}
Here $I_2^1$ is estimated as 
\begin{align*}
\begin{aligned}
I_2^1 &\leq C(1 - \delta_{k,1})\sum_{1 \leq l \leq k-1}\int_{\om} \lt( |\nabla n||\nabla^{k-1} v| + |\nabla^k v|\rt)|\nabla^l v| |\nabla^{k-l} n|\,dx\cr
&\quad + C\int_{\om} \lt( |\nabla n||\nabla^{k-1}v| + |\nabla^k v| \rt)|v||\nabla^k n|\,dx\cr
&\leq C(1 - \delta_{k,1})\sum_{1 \leq l \leq k-1}\lt( \epsilon_1 \|\nabla^{k-1} v\|_{L^2} + \|\nabla^k v\|_{L^2}\rt)\|\nabla^l v\|_{H^1}\|\nabla^{k-l} n\|_{H^1} \cr
&\quad + C\epsilon_1\lt(\|\nabla^{k-1} v\|_{L^2} + \|\nabla^k v\|_{L^2}\rt)\|\nabla^k n\|_{L^2}\cr
&\leq C\epsilon_1\|\nabla^k v\|_{L^2}^2 +C\epsilon_1\|\nabla^{k-1}v\|_{L^2}^2+ C\epsilon_1\|\nabla n\|_{H^{k-1}}^2.\cr
\end{aligned}
\end{align*}
In a similar way as the above, we find
\begin{align*}
\begin{aligned}
I_2^2 &\leq C(1 - \delta_{k,1})\sum_{1 \leq l \leq k-1}\int_{\om} \lt(|\nabla n||\nabla^{k-1} v| + |\nabla^k v|\rt)|\nabla^l n| |\nabla^{k-l}v|\,dx\cr
&\quad + C\int_{\om} \lt(|\nabla n||\nabla^{k-1} v| + |\nabla^k v|\rt)|1+n| |\nabla^k v|\,dx\cr
&\leq C(1 - \delta_{k,1})\sum_{1 \leq l \leq k-1}\lt(\epsilon_1\|\nabla^{k-1}v\|_{L^2} + \|\nabla^k v\|_{L^2}\rt)\|\nabla^l n\|_{H^1}\|\nabla^{k-l}v\|_{H^1} \cr
&\quad + C\epsilon_1\|\nabla^{k-1} v\|_{L^2}\|\nabla^k v\|_{L^2} + C\|\nabla^k v\|_{L^2}^2\cr
&\leq C\lt(\epsilon_1\|\nabla^{k-1} v\|_{L^2} + \|\nabla^k v\|_{L^2}\rt)\|\nabla v\|_{H^{k-1}}\|\nabla n\|_{H^{k-1}} + C\epsilon_1\|\nabla^{k-1}v\|_{L^2}^2+ C\|\nabla^k v\|_{L^2}^2\cr
&\leq C\epsilon_1\|\nabla^{k-1}v\|_{L^2}^2 + C\epsilon_1\|\nabla n\|_{H^{k-1}}^2+ C\|\nabla^k v\|_{L^2}^2,
\end{aligned}
\end{align*}
 Thus we have
\[
I_2 \leq C\epsilon_1\|\nabla^{k-1}v\|_{L^2}^2 + C\epsilon_1\|\nabla n\|_{H^{k-1}}^2+ C\|\nabla^k v\|_{L^2}^2.
\]
\noindent $\diamond$ Estimate of $I_3$: We split $I_3$ into five parts by using the momentum equations $\eqref{main_eq}_4$:
\begin{align*}
\begin{aligned}
I_3 &= -\frac{1}{2\mu + \lambda}\int_{\om} (1+n)^2\nabla^k n\cdot\lt(v \cdot \nabla^k v - [\nabla^{k-1}, v \cdot \nabla]v - \nabla^{k-1}\lt( \frac{e^h}{1+n}(u-v)\rt) \rt)dx\cr
&\quad -\frac{1}{2\mu + \lambda}\int_{\om}(1+n)^2\nabla^k n\cdot\lt( \nabla^{k-1}\lt( \frac{\nabla p(n+1)}{n+1}\rt) + \nabla^{k-1}\lt( \frac{Lv}{n+1}\rt) \rt)dx\cr
&=:\sum_{i=1}^5 I_3^i,
\end{aligned}
\end{align*}
where $I_3^i,i=1,2,3$ are estimated as follows.
\begin{align}\label{est_i1}
\begin{aligned}
I_3^1 &\ls \|v\|_{L^\infty}\|\nabla^k n\|_{L^2}\|\nabla^k v\|_{L^2} \leq C\epsilon_1\|\nabla^k n\|_{L^2}\|\nabla^k v\|_{L^2},\cr
I_3^2 &\ls \|\nabla v\|_{L^\infty}\|\nabla^k n\|_{L^2}\|\nabla^k v\|_{L^2} \leq C\epsilon_1\|\nabla^k n\|_{L^2}\|\nabla^{k} v\|_{L^2},\cr
I_3^3 &\ls \|\nabla^k n\|_{L^2}\lt\|\nabla^{k-1}\lt(\frac{e^h}{1+n}(u-v)\rt)\rt\|_{L^2} \cr
&\ls \|\nabla^k n\|_{L^2}\lt\| \nabla^{k-1}\lt( \frac{e^h}{1+n}\rt)\rt\|_{L^2}\|u - v\|_{L^\infty} + \|\nabla^k n\|_{L^2}\lt\| \nabla^{k-1}\lt( \frac{e^h}{1+n}\rt)\rt\|_{L^\infty}\|u - v\|_{L^2}\cr
&\ls \|\nabla^k n\|_{L^2}\|\nabla^{k-1}(h,n)\|_{L^2}\|(u ,v)\|_{L^\infty} + \|\nabla^k n\|_{L^2}\|\nabla^{k-1}(u,v)\|_{L^2}\cr
&\leq C\epsilon_1\|\nabla^k n\|_{L^2}^2 + C\epsilon_1\|\nabla^{k-1}(h,n)\|_{L^2}^2 + \delta_2\|\nabla^k n\|_{L^2}^2 + C_{\delta_2}\|\nabla^{k-1}(u,v)\|_{L^2}^2,
\end{aligned}
\end{align}
where $\delta_2 > 0$ will be chosen later. For the estimate of $I_3^4$, we obtain
\begin{align*}
\begin{aligned}
I_3^4 &= - \frac{\gamma}{2\mu + \lambda}\int_{\om} (1+n)^2\nabla^k n \cdot\nabla^{k-1}\lt( (1+n)^{\gamma-2}\nabla n \rt)dx\cr
&=-\frac{\gamma(1 - \delta_{k,1})}{2\mu + \lambda}\sum_{1\leq l \leq k-1}\binom{k-1}{l}\int_{\om}(1+n)^2 \nabla^k n \cdot \lt(\nabla^l\lt( (1+n)^{\gamma-2}\rt)\nabla^{k-l}n\rt)dx\cr
&\quad - \frac{\gamma}{2\mu + \lambda}\int_{\om}(1+n)^\gamma |\nabla^k n|^2 dx\cr
&=: I_3^{4,1} + I_3^{4,2}.
\end{aligned}
\end{align*}
We estimate $I_3^{4,1}$ as
\begin{align*}
\begin{aligned}
I_3^{4,1} &\ls (1 - \delta_{k,1})\sum_{1\leq l \leq k-1}\|\nabla^k n\|_{L^2}\|\nabla^l (1+n)^{\gamma-2}\|_{L^4}\|\nabla^{k-l}n\|_{L^4}\cr
&\ls \|\nabla^k n\|_{L^2}\|\nabla (1+n)^{\gamma-2}\|_{H^{k-1}}\|\nabla n\|_{H^{k-1}}\cr
&\leq C\epsilon_1\|\nabla^k n\|_{L^2}\|\nabla n\|_{H^{k-1}}\cr
&\leq C\epsilon_1\|\nabla n\|_{H^{k-1}}^2,
\end{aligned}
\end{align*}
where we used the Sobolev inequality in Lemma \ref{lem:bern} to get $\|\nabla (1+n)^{\gamma-2}\|_{H^{k-1}} \leq C\epsilon_1$.
For the $I_3^{4,2}$, we use $\|n\|_{L^\infty }\leq \epsilon_1 \ll 1$ to get
\[
I_3^{4,2} \leq - \frac{C\gamma}{2\mu + \lambda}\|\nabla^k n\|_{L^2}^2,
\]
This yields
\bq\label{est_i2}
I_3^4 \leq C\epsilon_1\|\nabla n\|_{H^{k-1}}^2 - \frac{C\gamma}{2\mu + \lambda}\|\nabla^k n\|_{L^2}^2.
\eq
We divide $I_3^5$ into three terms by using the Leibniz rule:
\begin{align*}
\begin{aligned}
I_3^5 &= \frac{1}{2\mu + \lambda}\int_{\om} (1+n)^2\nabla^k n \cdot\lt( \frac{\mu\nabla^{k-1}(\nabla \cdot \nabla v)}{n+1} + \frac{\mu + \lambda}{n+1}\nabla^{k-1}\nabla \nabla \cdot v\rt) dx\cr
& \quad + \frac{1 - \delta_{k,1}}{2\mu+\lambda}\sum_{1\leq l \leq k-1}\binom{k-1}{l}\int_{\om} (1+n)^2 \nabla^k n \nabla^l\lt( \frac{\mu}{n+1}\rt) \nabla^{k-1-l}(\nabla \cdot \nabla v)\,dx\cr
&\quad + \frac{1 - \delta_{k,1}}{2\mu + \lambda}\sum_{1\leq l \leq k-1}\binom{k-1}{l}\int_{\om} (1+n)^2 \nabla^k n \nabla^l\lt( \frac{\mu + \lambda}{n+1}\rt) \nabla^{k-1-l}(\nabla \nabla \cdot v)\,dx\cr
&=:\sum_{i=1}^3 I_3^{5,i},
\end{aligned}
\end{align*}
where $I_3^{5,1}$ is estimated as follows.
\begin{align*}
\begin{aligned}
I_3^{5,1} &= \frac{\mu}{2\mu + \lambda}\int_{\om}(1+n)\nabla^k n \cdot \nabla^{k-1}(\nabla\cdot \nabla v)dx+ \frac{\mu+\lambda}{2\mu + \lambda}\int_{\om} (1+n)\nabla^k n \cdot \nabla^{k-1}(\nabla \nabla \cdot v) dx\cr
&=-\frac{\mu}{2\mu + \lambda}\int_{\om} \lt( \pa_j n \nabla^{k-1} \pa_i n + (1+n)\nabla^{k-1}\pa_{ij}n \rt) \nabla^{k-1}\pa_j v_i\,dx\cr
&\quad - \frac{\mu + \lambda}{2\mu + \lambda}\int_{\om} \lt( \pa_j n \nabla^{k-1}\pa_i n + (1+n)\nabla^{k-1}\pa_{ij}n \rt)\nabla^{k-1}\pa_i v_j\,dx\cr
&= -\frac{1}{2\mu + \lambda}\int_{\om} \lt( \mu \pa_j n \nabla^{k-1}\pa_i n + (\mu + \lambda)\pa_i n \nabla^{k-1}\pa_j n\rt) \nabla^{k-1}\pa_j v_i\, dx\cr
&\quad - \int_{\om} (1+n)\nabla^{k-1}\pa_{ij} n \cdot \nabla^{k-1}\pa_j v_i\, dx\cr
&\leq C\|\nabla^k n\|_{L^2}\|\nabla^k v\|_{L^2}\|\nabla n\|_{L^\infty} - \int_{\om} (1+n)\nabla^{k-1}\pa_{ij} n \cdot \nabla^{k-1}\pa_j v_i\, dx\cr
&\leq C\epsilon_1\|\nabla^k (n,v)\|_{L^2}^2 - \int_{\om} (1+n)\nabla^{k-1}\pa_{ij} n \cdot \nabla^{k-1}\pa_j v_i\, dx.
\end{aligned}
\end{align*}
We remind the reader that the summation notation is omitted. We also estimate $I_3^{5,i},i=2,3$ as 
\begin{align*}
\begin{aligned}
I_3^{5,2} + I_3^{5,3} &\ls (1 - \delta_{k,1})\sum_{1\leq l \leq k-1}\|\nabla^k n\|_{L^2}\lt\| \nabla^l\lt( \frac{1}{n+1}\rt)\rt\|_{L^4}\|\nabla^{k+1-l} v\|_{L^4}\cr
&\ls \|\nabla^k n\|_{L^2}\|\nabla n\|_{H^{k-1}}\|\nabla^2 v\|_{H^{k-1}}\cr
&\leq C\epsilon_1\|\nabla^k n\|_{L^2}\|\nabla^2 v\|_{H^{k-1}}\cr
&\leq C\epsilon_1\|\nabla^k n\|_{L^2}^2 + C\epsilon_1\|\nabla^2 v\|_{H^{k-1}}^2.
\end{aligned}
\end{align*}
This implies
\bq\label{est_i3}
I_3^5 \leq C\epsilon_1\|\nabla^k n\|_{L^2}^2 + C\epsilon_1\|\nabla v\|_{H^k}^2- \int_{\om} (1+n)\nabla^{k-1}\pa_{ij} n \cdot \nabla^{k-1}\pa_j v_i\, dx.
\eq
We now collect the estimates \eqref{est_i1}-\eqref{est_i3} to have
$$\begin{aligned}
I_3 &\leq C\epsilon_1\|\nabla n\|_{H^{k-1}}^2 + C\epsilon_1\|\nabla v\|_{H^k}^2 + C\epsilon_1\|\nabla^{k-1}(h,n)\|_{L^2}^2 + C\delta_2\|\nabla^{k-1}(u,v)\|_{L^2}^2\cr
& - \lt(\frac{C\gamma}{2\mu+\lambda} - \delta_2\rt)\|\nabla^k n\|_{L^2}^2- \int_{\om} (1+n)\nabla^{k-1}\pa_{ij} n \cdot \nabla^{k-1}\pa_j v_i\, dx.
\end{aligned}$$
$\diamond$ Estimate of $I_4$: Since
\[
I_4 = -\frac{2}{2\mu + \lambda}\int_{\om} (1+n)\nabla^{k-1} n \cdot \nabla^{k-1}v \lt( \nabla n \cdot v + (1+n)\nabla \cdot v\rt)dx,
\]
we easily obtain
\begin{align*}
\begin{aligned}
I_4 &\ls \|\nabla^{k-1}n\|_{L^2}\|\nabla^{k-1}v\|_{L^2}\lt( \|\nabla n \cdot v\|_{L^\infty} + \|\nabla v\|_{L^\infty}\rt) \leq C\epsilon_1\|\nabla^{k-1}(n,v)\|_{L^2}^2.
\end{aligned}
\end{align*}
We finally combine all of the above estimates to find
\begin{align*}
\begin{aligned}
&\frac{d}{dt}\int_{\om} \nabla^k n \cdot \lt( \frac{\nabla^k n}{2} + \frac{(n+1)^2}{2\mu + \lambda} \nabla^{k-1} v\rt) dx + \lt(\frac{C\gamma}{2\mu + \lambda} - \delta_2 - C\epsilon_1\rt)\|\nabla^k n\|_{L^2}^2\cr
&\qquad \leq C\epsilon_1\|n\|_{H^{k-1}}^2 + C\epsilon_1\|v\|_{H^{k+1}}^2 + C\epsilon_1\|\nabla^{k-1}h\|_{L^2}^2 + C\|\nabla^k v\|_{L^2}^2 + C_{\delta_2}\|\nabla^{k-1}(u,v)\|_{L^2}^2,
\end{aligned}
\end{align*}
and by choosing $\delta_2 > 0$ small enough such that $\frac{C\gamma}{2\mu + \lambda} - \delta_2 - C\epsilon_1 > 0$, we have that for $1 \leq k \leq s$
\begin{align*}
\begin{aligned}
&\frac{d}{dt}\int_{\om} \nabla^k n \cdot \lt( \frac{\nabla^k n}{2} + \frac{(n+1)^2}{2\mu + \lambda} \nabla^{k-1} v\rt) dx + C_1\|\nabla^k n\|_{L^2}^2\cr
&\quad \leq C\epsilon_1\|n\|_{H^{k-1}}^2 + C\epsilon_1\|v\|_{H^{k+1}}^2 + C\epsilon_1\|\nabla^{k-1}h\|_{L^2}^2 + C\|\nabla^k v\|_{L^2}^2 + C\|\nabla^{k-1}(u,v)\|_{L^2}^2
\end{aligned}
\end{align*}
for some positive constant $C_1 > 0$. Here $C$ is a positive constant independent of $t$.

%
%
%
%


\begin{thebibliography}{10}
\bibitem{BD} C. Baranger and L. Desvillettes, Coupling Euler and Vlasov equations in the context of sprays: the local-in-time, classical solutions, J. Hyper. Diff., Eqns., 3, (2006), 1--26.
\bibitem{BBKT} S. Berres, R. B\"urger, K. H. Karlsen, and E. M. Tory, Strongly degenerate parabolic-hyperbolic systems modeling polydisperse sedimentation with compressible, SIAM Appl. Math., 64, (2003), 41--80.
\bibitem{CCK} J. A. Carrillo, Y.-P. Choi, and T. K. Karper, On the analysis of a coupled kinetic-fluid model with local alignment forces, Ann. I. H. Poincar\'e - AN, 33, (2016), 273--307.
\bibitem{Choi1} Y.-P. Choi, Compressible Euler equations interacting with incompressible flow, Kinetic and Related Models, 8, (2015), 335--358.
\bibitem{Choi2} Y.-P. Choi, Large-time behavior for the Vlasov/compressible Navier-Stokes equations, to appear in J. Math. Phys.
\bibitem{Choi3} Y.-P. Choi, Global classical solutions of the Vlasov-Fokker-Planck equation with local alignment forces, Nonlinearity, 29, (2016), 1887--1916.
\bibitem{CK} Y.-P. Choi and B. Kwon, The Cauchy problem for the pressureless/isentropic Navier-Stokes equations, J. Differential Equations, 261, (2016), 654--711.
\bibitem{DL} R. Duan and S. Liu, Cauchy problem on the Vlasov-Fokker-Planck equation coupled with the compressible Euler equations through the friction force, Kinetic and Related Models, 6, (2013), 687--700.
\bibitem{Gal} G. P. Galdi, An introduction to the mathematical theory of the Navier-Stokes equations I, Springer-Verlag, New York, 1994.
\bibitem{HKK} S.-Y. Ha, M.-J. Kang and B. Kwon, A hydrodynamic model for the interaction of Cucker-Smale particles and incompressible fluids, Math. Mod. Meth. in Appl. Sci., 24, (2014), 2311--2359.
\bibitem{KMT} T. Karper, A. Mellet, and K. Trivisa, Existence of weak solutions to kinetic flocking models, SIAM Math. Anal., 45, (2013), 215--243.
\bibitem{KMT2} T. Karper, A. Mellet, and K. Trivisa, Hydrodynamic limit of the kinetic Cucker-Smale flocking model, Math. Models Meth. Appl. Sci., 25, (2015), 131--163. 
\bibitem{Majda} A. Majda, Introduction to PDEs and Waves for the Atmosphere and Ocean, Courant Lecture Notes in Mathematics, 9, AMS/CIMS, (2003).  
\bibitem{MN1} A. Matsumura and T. Nishida, The initial value problem for the equations of motion of viscous and heat-conductive gases, J. Math. Kyoto Univ., 20, (1980), 67Ð104.
\bibitem{MN2} A. Matsumura and T. Nishida, Initial boundary value problems for the equations of motion of compressible viscous and heat-conductive fluids, Comm. Math. Phys., 89, (1983), 445Ð464.
\bibitem{MV} A. Mellet and A. Vasseur, Global weak solutions for a Vlasov-Fokker-Planck/Navier-Stokes system of equations, Math. Mod. Meth. in Appl. Sci., 17, (2007), 1039--1063.
\bibitem{PZ} R. Pan and K. Zhao, The 3D compressible Euler equations with damping in a bounded domain, J. Diff. Eqns., 246, (2009), 581--596.
\bibitem{Sar} W. K. Sartory, Three-component analysis of blood sedimentation by the method of characteristics, Math. Biosci., 33, (1977), 145--165.
\bibitem{STW} T. C. Sideris, B. Thomases and D. Wang, Long time behavior of solutions to the 3D compressible Euler equations with damping, Comm. Partial Differential Equations, 28, (2003), 795--816.
\bibitem{Will} F. A. Williams, Spray combustion and atomization, Phys. Fluids, 1, (1958), 541--555.


\end{thebibliography}
\end{document}